\documentclass[12pt,english]{article}
\usepackage[T1]{fontenc}
\usepackage[latin9]{inputenc}
\usepackage[a4paper]{geometry}
\usepackage{mathrsfs}
\usepackage{amsmath}
\usepackage{amsthm}
\usepackage{amssymb}

\makeatletter

\providecommand{\tabularnewline}{\\}

\numberwithin{equation}{section}
\numberwithin{figure}{section}
\theoremstyle{remark}
\newtheorem*{acknowledgement*}{\protect\acknowledgementname}
\theoremstyle{plain}
\newtheorem{thm}{\protect\theoremname}[section]
\theoremstyle{plain}
\newtheorem{lem}[thm]{\protect\lemmaname}
\theoremstyle{plain}
\newtheorem{prop}[thm]{\protect\propositionname}
\theoremstyle{plain}
\newtheorem{cor}[thm]{\protect\corollaryname}
\theoremstyle{plain}
\newtheorem{assumption}[thm]{\protect\assumptionname}
\theoremstyle{definition}
\newtheorem*{example*}{\protect\examplename}
\theoremstyle{definition}
\newtheorem{defn}[thm]{\protect\definitionname}
\theoremstyle{plain}
\newtheorem*{thm*}{\protect\theoremname}
\theoremstyle{plain}
\newtheorem*{prop*}{\protect\propositionname}
\theoremstyle{remark}
\newtheorem{rem}[thm]{\protect\remarkname}

\date{}

\makeatother

\usepackage{babel}
\providecommand{\acknowledgementname}{Acknowledgement}
\providecommand{\assumptionname}{Assumption}
\providecommand{\corollaryname}{Corollary}
\providecommand{\definitionname}{Definition}
\providecommand{\examplename}{Example}
\providecommand{\lemmaname}{Lemma}
\providecommand{\propositionname}{Proposition}
\providecommand{\remarkname}{Remark}
\providecommand{\theoremname}{Theorem}

\begin{document}
\title{Capelli identity and contiguity relations of Radon hypergeometric
function on the Grassmannian}
\author{Hironobu Kimura,\\
 Department of Mathematics, Graduate School of Science and\\
 Technology, Kumamoto University}

\maketitle

\global\long\def\R{\mathbb{R}}%
 
\global\long\def\al{\alpha}%
\global\long\def\be{\beta}%
 
\global\long\def\ga{\gamma}%
 
\global\long\def\de{\delta}%
 
\global\long\def\expo{\mathrm{exp}}%
 
\global\long\def\f{\varphi}%
 
\global\long\def\W{\Omega}%
 
\global\long\def\wm{\omega}%
 
\global\long\def\lm{\lambda}%
\global\long\def\te{\theta}%
 
\global\long\def\C{\mathbb{C}}%
 
\global\long\def\Z{\mathbb{Z}}%
 
\global\long\def\Ps{\mathbb{P}}%
 
\global\long\def\De{\Delta}%
 
\global\long\def\cbatu{\mathbb{C}^{\times}}%
 
\global\long\def\La{\Lambda}%
 
\global\long\def\vt{\vartheta}%
 
\global\long\def\ep{\epsilon}%
 
\global\long\def\G{\Gamma}%
\global\long\def\GL#1{\mathrm{GL}(#1)}%
 
\global\long\def\gras{\mathrm{Gr}}%
 
\global\long\def\fl{\mathrm{Flag}}%
 
\global\long\def\vep{\varepsilon}%
 
\global\long\def\Gl{\mathrm{GL}}%
 
\global\long\def\diag{\mathrm{diag}}%
 
\global\long\def\tr{\,\mathrm{^{t}}}%
 
\global\long\def\re{\mathrm{Re}}%
 
\global\long\def\im{\mathrm{Im}}%
 
\global\long\def\sm{\sigma}%
 
\global\long\def\ini{\mathrm{in}_{\prec}}%
 
\global\long\def\Span{\mathrm{span}}%
 
\global\long\def\cS{\mathcal{S}}%
 
\global\long\def\cL{\mathcal{L}}%
 
\global\long\def\rank{\mathrm{rank}}%
 
\global\long\def\tH{\tilde{H}}%
 
\global\long\def\mat{\mathrm{Mat}}%
 
\global\long\def\lto{\longrightarrow}%
 
\global\long\def\Si{\mathfrak{S}}%
 
\global\long\def\cO{\mathcal{O}}%
 
\global\long\def\gl{\mathfrak{gl}}%
\global\long\def\gee{\mathfrak{g}}%
 
\global\long\def\Tr{\,\mathrm{Tr}}%
 
\global\long\def\ad{\mathrm{ad}}%
 
\global\long\def\Ad{\mathrm{Ad}}%
 
\global\long\def\ha{\mathfrak{h}}%
 
\global\long\def\fj{\mathfrak{j}}%
 
\global\long\def\mrn{\mat'(r,N)}%
 
\global\long\def\sgn{\mathrm{sgn}}%
 
\global\long\def\hlam{H_{\lambda}}%
 
\global\long\def\mnm{\mat'(m,N)}%
\global\long\def\cP{\mathcal{P}}%
 
\global\long\def\ghyp{\,_{2}F_{1}}%
 
\global\long\def\kum{\,_{1}F_{1}}%
 
\global\long\def\auto{\mathrm{Aut}}%
 
\global\long\def\la{\langle}%
 
\global\long\def\ra{\rangle}%
 
\global\long\def\Ai{\mathrm{Ai}}%
 
\global\long\def\adj{\mathrm{Ad}}%
 
\global\long\def\yn{\mathbf{Y}_{n}}%
 
\global\long\def\eq{\mathcal{I}}%
 
\global\long\def\bl{B_{\lambda}}%
 
\global\long\def\hyp#1#2{\, _{#1}F_{#2}}%
 
\global\long\def\jro{J_{r}^{\circ}}%
 
\global\long\def\jroo{\mathfrak{j}_{r}^{\circ}}%
\global\long\def\jor#1{J^{\circ}(#1)}%
 
\global\long\def\fa{\mathfrak{a}}%
 
\global\long\def\pa{\partial}%
 
\global\long\def\bx{\mathbf{x}}%
 
\global\long\def\herm{\mathscr{H}(r)}%
 
\global\long\def\etr{\mathrm{etr}}%
 
\global\long\def\cR{\mathcal{R}}%
 
\global\long\def\bb{\mathbf{b}}%

\begin{abstract}
The Radon hypergeometric function (Radon HGF) $F(z;\al)$ of type
$\lm$ on the Grassmannian $\gras(m,N)$, $N=rn$ for some integers
$r,n>0$, is defined as a Radon transform of the character of the
universal covering $\tilde{H}_{\lm}$ of the Lie group $H_{\lm}\subset\GL N$
specified by a partition $\lm$ of $n$, where $\al\in\C^{n}$ is
the parameter to in the character. For this Radon HGF, we give the
contiguity relations of the form $\cL^{(i,j)}F(z;\al)=\bb(\al_{0}^{(j)})F(z;\al+\ep^{(i)}-\ep^{(j)})$
with a differential operator $\cL^{(i,j)}$of order $r$ and the polynomial
$\bb(s)$ called $b$-function. As its application, we derive the
contiguity relation for the beta and gamma functions defined by Hermitian
matrix integral. In establishing the contiguity relations for Radon
HGF, the classical Capelli identity and Cayley's formula play an essential
role.
\end{abstract}

\section{Introduction}

In \cite{kimura-2}, we introduced the Radon hypergeometric function
(Radon HGF) of non-confluent and confluent type on the Grassmannian
manifold $\gras(m,N)$, with $N=rn$ for some $1\leq r<m$ and $n$,
which gives an extension of Gelfand's HGF \cite{Gelfand,Kimura-Haraoka}
(the case $r=1$ is the Gelfand HGF). We also discussed in \cite{kimura-2}
a relation of Radon HGF to HGFs defined by Hermitian matrix integrals. 

The Radon HGF is defined as a Radon transform of a character $\chi_{\lm}$
of the universal covering group of subgroup $H_{\lm}\subset\GL N$
indexed by a partition $\lm$ of $n$. For example, when $\lm=(1,\dots,1)$,
$H_{\lm}=\{\diag(h^{(1)},\dots,h^{(n)})\in\GL N\mid h^{(k)}\in\GL r\;(1\leq k\leq n)\}$
and $\chi_{\lm}=\prod_{1\leq k\leq n}(\det h^{(k)})^{\al^{(k)}},\al^{(k)}\in\C$,
and the Radon HGF has the form 
\[
F(z;\al)=\int_{C(z)}\prod_{1\leq k\leq n}(\det\vec{u}z^{(k)})^{\al^{(k)}}du
\]
with independent variables $z=(z^{(1)},\dots,z^{(n)}),z^{(k)}\in\mat(m,r)$,
$\vec{u}=(1_{r},u)$ with $u=(u_{a,b})\in\mat(r,m-r)$ and $du=\wedge du_{a,b}$.
Note that HGFs by Hermitian matrix integral are important and are
studied in various fields of mathematics and statistics \cite{Faraut,inamasu-ki,Kontsevich,Mehta,muirhead,muirhead-2,shimura}.
The purpose of this paper is to give the contiguity relations for
the Radon HGF and to derive the known contiguity relations for the
beta and gamma functions defined by Hermitian matrix integral as an
application of our result. See \cite{Faraut,inamasu-ki,muirhead-2}
for the beta and gamma function by Hermitian matrix integral. For
the Gelfand HGF, which is a particular case of Radon HGF, we know
several works on the contiguity relations \cite{Horikawa,Kimura-H-T,mimachi,sasaki}.
So the contiguity relations for the Radon HGF give an extension of
those for the Gelfand HGF.

Since the Gelfand HGF (and hence the Radon HGF) contains the Gauss
HGF as one of the simplest cases of non-confluent type, let us explain
what does ``contiguity relations'' mean by taking the Gauss HGF
as an example. The Gauss HGF is defined by the series
\begin{equation}
\hyp 21(a,b,c;x)=\sum_{k=0}^{\infty}\frac{(a)_{k}(b)_{k}}{(c)_{k}k!}x^{k},\label{eq:intro-1}
\end{equation}
where $a,b,c\in\C$ with $c\notin\Z_{\leq0}$ and $(a)_{k}=\G(a+k)/\G(a)$
is the Pochhammer symbol defined by the gamma function $\G(a)$. This
series gives a holomorphic function in the unit disc $|x|<1$ of $\C$.
To continue $\hyp 21(a,b,c;x)$ analytically in $x$ beyond the unit
circle $|x|=1$, the following integral representation for the series
(\ref{eq:intro-1}) is useful: 
\begin{equation}
\frac{\G(c)}{\G(a)\G(c-a)}\int_{0}^{1}u^{a-1}(1-u)^{c-a-1}(1-ux)^{-b}du.\label{eq:intro-2}
\end{equation}
In fact, by deforming the path of integration $\overrightarrow{0,1}$
in (\ref{eq:intro-2}) in the complex $u$-plane suitably, we can
continue it analytically with respect to $x$ along any path in $\C\setminus\{0,1\}$
starting from any point in the disc $|x|<1$. However in this case
we must assume $\re\,a>0,\re(c-a)>0$ to assure the convergence of
the integral. The series (\ref{eq:intro-1}) satisfies the differential
equation
\begin{equation}
[x(1-x)\pa^{2}+\left\{ c-(a+b+1)x\right\} \pa-ab]y=0,\quad\pa=d/dx\label{eq:intro-3}
\end{equation}
for (\ref{eq:intro-1}), in fact, for $a,b,c$ generic, $\hyp 21(a,b,c;x)$
is characterized as a unique holomorphic solution $y(x)$ to (\ref{eq:intro-3})
at $x=0$ with $y(0)=1$. Contiguity relations for $\hyp 21(a,b,c;x)$
are formulas which connect $\hyp 21(a,b,c;x)$ to $\hyp 21(a',b',c';x)$,
where $(a',b',c')$ is obtained from $(a,b,c)$ by increasing or decreasing
one of $a,b,c$ by $1$. For example we have 
\begin{gather}
(x\pa+a)\hyp 21(a,b,c;x)=a\cdot\hyp 21(a+1,b,c;x),\label{eq:intro-4-1}\\
(x(1-x)\pa+c-a-bx)\hyp 21(a,b,c;x)=(c-a)\cdot\hyp 21(a-1,b,c;x),\label{eq:intro-4}
\end{gather}
see \cite{IKSY}. The differential operators which give these relations
are called the \emph{contiguity operators}. We should point out that
the contiguity relations above hold for any solution to (\ref{eq:intro-3})
and that any solution of (\ref{eq:intro-3}) can be expressed by the
integral (\ref{eq:intro-2}) taking an appropriate path of integration
instead of $\overrightarrow{0,1}$. The importance of contiguity relations
can be explained as follows. For $\hyp 21(a,b,c;x)$, which is identified
with the integral (\ref{eq:intro-2}), consider its analytic continuation
with respect to the parameter $a$. By the expression (\ref{eq:intro-2}),
the analytic continuation in $x$ becomes possible, but to ensure
the convergence of the integral we assumed $\re\,a>0,\re(c-a)>0$.
Using (\ref{eq:intro-4}), we can carry out the analytic continuation
with respect to $a$. In fact, put $a\mapsto a+1$ in (\ref{eq:intro-4})
and obtain 
\[
(x(1-x)\pa+c-a-1-bx)\hyp 21(a+1,b,c;x)=(c-a-1)\hyp 21(a,b,c;x).
\]
The left hand side is analytic in $a$ for $\re\,a>-1$. It follows
that $\hyp 21(a,b,c;x)$ is analytically continued with respect to
$a$ to a larger domain $\re\,a>-1$ with a possible pole at $c-a-1=0$.

We know that the Gelfand HGF provides a framework to understand many
classical HGFs from a unified viewpoint \cite{Gelfand,Kimura-Haraoka}.
For example, the beta and gamma functions are understood as Gelfand's
HGF on $\gras(2,3)$ corresponding to partitions $(1,1,1)$ and $(2,1)$,
respectively, and the Gauss HGF and its confluent family, namely,
Kummer's confluent HGF, Bessel function, Hermite-Weber function and
Airy function, are understood as the Gelfand HGF on $\gras(2,4)$
corresponding to the partitions $(1,1,1,1)$, $(2,1,1)$, $(2,2)$,
$(3,1)$ and $(4)$, respectively. See \cite{Appell-2,Erdelyi,IKSY}
for the these classical HGF. The Radon HGF plays a similar role for
HGFs defined by Hermitian matrix integral. The beta and gamma functions
by Hermitian martix integral are defined as 
\begin{align*}
B_{r}(a,b) & :=\int_{0<U<1_{r}}(\det U)^{a-r}(\det(1_{r}-U))^{b-r}\,dU,\\
\G_{r}(a) & :=\int_{U>0}\etr(-U)(\det U)^{a-r}\,dU,
\end{align*}
where $U$ is an integration variable in the space $\herm$ of Hermitian
matrices of size $r$, $dU$ is the standard Euclidean volume form
of a real vector space $\herm$ of dimesion $r^{2}$ and $0<U<1_{r}$
implies that $U,1_{r}-U$ are positive definite. We also have the
Hermitian matrix integral analogue of the Gauss and its confluent
family \cite{Faraut,inamasu-ki}. For example 
\begin{gather}
\frac{\G_{r}(c)}{\G_{r}(a)\G_{r}(c-a)}\int_{0<U<1_{r}}(\det U)^{a-r}(\det(1_{r}-U))^{c-a-r}(\det(1_{r}-UX))^{-b}\,dU,\label{eq:intro-5}\\
\frac{\G_{r}(c)}{\G_{r}(a)\G_{r}(c-a)}\int_{0<U<1_{r}}(\det U)^{a-r}(\det(1_{r}-U))^{c-a-r}\exp(\Tr(UX))\,dU\label{eq:intro-6}
\end{gather}
are analogues of Gauss and Kummer, and they can be understood as the
Radon HGF on $\gras(2r,4r)$ corresponding to the partitions $(1,1,1,1)$
and $(2,1,1)$, respectively \cite{kimura-2}.

Our contiguity relation for the Radon HGF may serve the study of analytic
continuation of the HGFs by Hermitian matrix integral with respect
to the parameters, the parameters $a,b,c$ in the case (\ref{eq:intro-5}),
for example. In fact, in Section \ref{sec:Application-to-gamma} we
give the formula
\begin{align}
B_{r}(a+1,b) & =\frac{a(a-1)\cdots(a-r+1)}{(a+b)(a+b-1)\cdots(a+b-r+1)}B_{r}(a,b),\label{eq:intro-7}\\
\G_{r}(a+1) & =a(a-1)\cdots(a-r+1)\G_{r}(a)\label{eq:intro-8}
\end{align}
as an application of Theorems \ref{thm:conti-nonconf} and \ref{thm:cont-conf}.
Application to the Hermitian matrix integral analogues of the Gauss
HGF and its confluent family will be discussed in another paper. Note
that an analogue of (\ref{eq:intro-4-1}) for (\ref{eq:intro-5})
is already given in Proposition XV.3.1 of \cite{Faraut}. Note also
that the study of analytic continuation with respect to $a,c$ for
(\ref{eq:intro-6}) is important in the problem of  number theory
and discussed in \cite{shimura}.

In establishing contiguity relations for the Radon HGF, the key point
is Cayley's formula in the invariant theory: for $u=(u_{i,j})_{1\leq i,j\leq r}$
and $\pa_{i,j}=\pa/\pa u_{i,j}$, we have
\[
\det(\pa_{i,j})(\det u)^{s}=\bb(s)(\det u)^{s-1},\quad\bb(s)=s(s+1)\cdots(s+r-1),
\]
where $\bb(s)$ is called the $b$-function in the theory of prehomogeneous
vector spaces. The Cayley formula is obtained from the classical Capelli
identity, and when $r=1$, it reduces to almost trivial formula
\begin{equation}
\frac{d}{du}u^{s}=su^{s-1}\label{eq:intro-9}
\end{equation}
and we can obtain the famous identity $\G(s+1)=s\G(s)$ by multiplying
the both sides of (\ref{eq:intro-9}) by $e^{-u}$ and integrating
them on the interval $(0,\infty)$. 

This paper is organized as follows. In Section 2, we recall the definition
of the Radon HGF. In Section 3, after recalling the facts on the Capelli
identity and Cayley's formula, we state our main results, Theorem
\ref{thm:conti-nonconf} and Theorem \ref{thm:cont-conf}, which give
the explicit form of the contiguity relations for the Radon HGF of
non-confluent type and confluent type, respectively. The proofs of
the theorems are given in Sections \ref{subsec:Proof-nonconf} and
\ref{subsec:Proof-conf}. In Section \ref{sec:Application-to-gamma},
we apply Theorems \ref{thm:conti-nonconf} and \ref{thm:cont-conf}
to the Radon HGF corresponding to the beta and gamma functions defined
by Hermitian matrix integrals, and we derive (\ref{eq:intro-7}) and
(\ref{eq:intro-8}) from the contiguity relations for the corresponding
Radon HGF. 
\begin{acknowledgement*}
This work was supported by JSPS KAKENHI Grant Number JP19K03521.
\end{acknowledgement*}

\section{\label{sec:Radon-HGF}Radon HGF}

\subsection{Jordan group}

We recall the definition of Radon HGF. For the detailed explanation,
see \cite{kimura-2}. Let $r$ and $N$ be positive integers such
that $r<N$ and assume $N=nr$ for some integer $n$. Suppose we are
given a partition $\lm=(n_{1},n_{2},\dots,n_{\ell})$ of $n$, namely
a nonincreasing sequence of positive integers $n_{1}\geq n_{2}\geq\cdots\geq n_{\ell}$
such that $|\lm|:=n_{1}+\cdots+n_{\ell}=n$. For such $\lm$, let
us consider a complex Lie subgroup $H_{\lm}$ of the complex general
linear group $G=\GL N$. Put 

\[
J_{r}(p):=\left\{ h=\left(\begin{array}{cccc}
h_{0} & h_{1} & \dots & h_{p-1}\\
 & \ddots & \ddots & \vdots\\
 &  & \ddots & h_{1}\\
 &  &  & h_{0}
\end{array}\right)\mid h_{0}\in\GL r,\ h_{i}\in\mat(r)\right\} \subset\GL{pr},
\]
which is a Lie group called the (generalized) Jordan group. We define
\[
H_{\lm}:=\left\{ h=\diag(h^{(1)},\dots,h^{(\ell)})\mid h^{(j)}\in J_{r}(n_{j})\right\} \subset G.
\]
Then $H_{\lm}\simeq J_{r}(n_{1})\times\cdots\times J_{r}(n_{\ell})$,
where an element $(h^{(1)},\dots,h^{(\ell)})\in J_{r}(n_{1})\times\cdots\times J_{r}(n_{\ell})$
is identified with a block diagonal matrix $\diag(h^{(1)},\dots,h^{(\ell)})\in H_{\lm}$.
In particular, for $\lm=(1,\dots,1)$, $H_{\lm}\simeq(\GL r)^{n}$
and, when $r=1$ it reduces to the Cartan subgroup of $G$ consisting
of diagonal matrices. We also use a unipotent subgroup $\jro(p)\subset J_{r}(p)$
defined by 
\[
\jro(p):=\left\{ h=\left(\begin{array}{cccc}
1_{r} & h_{1} & \dots & h_{p-1}\\
 & \ddots & \ddots & \vdots\\
 &  & \ddots & h_{1}\\
 &  &  & 1_{r}
\end{array}\right)\mid h_{i}\in\mat(r)\right\} .
\]
An element $h\in J_{r}(p)$ is expressed as $h=\sum_{0\leq i<p}h_{i}\otimes\La^{i}$
using the shift matrix $\La=(\de_{i+1,j})$ of size $p$. Taking into
account the expression $h=\sum_{0\leq i<p}h_{i}\otimes\La^{i}$, we
can describe $J_{r}(p)$ and $\jro(p)$ as follows. Put $R=\mat(r)$
and consider it as a $\C$-algebra. Let $R[w]$ be the ring of polynomials
in $w$ with coefficients in $R$. Then $J_{r}(p)$ is identified
with the group of units in the quotient ring of $R[w]$ by the principal
ideal $(w^{p})$:
\[
J_{r}(p)\simeq\left(R[w]/(w^{p})\right)^{\times}.
\]
Thus we can write $J_{r}(p)$ and $\jro(p)$ as 
\begin{align*}
J_{r}(p) & \simeq\left\{ h_{0}+\sum_{1\leq i<p}h_{i}w^{i}\in R[w]/(w^{p})\mid h_{0}\in\GL r\right\} ,\\
\jro(p) & \simeq\left\{ 1_{r}+\sum_{1\leq i<p}h_{i}w^{i}\in R[w]/(w^{p})\right\} .
\end{align*}
In the following we use freely this identification when necessary. 

\subsection{\label{subsec:Char-conf-1}Character of Jordan group}

In this section, we give the characters of the universal covering
group of Jordan group and of $H_{\lm}$. The Radon HGF is defined
as a Radon transform of these characters. The following lemma is easily
shown.
\begin{lem}
\label{lem:radon-conf-1}We have a group isomorphism 
\[
J_{r}(p)\simeq\GL r\ltimes\jro(p)
\]
defined by the correspondence $\GL r\ltimes\jro(p)\ni(g,h)\mapsto g\cdot h=\sum_{0\leq i<p}(gh_{i})\otimes\La^{i}\in J_{r}(p)$,
where the semi-direct product is defined by the action of $\GL r$
on $\jro(p)$: $h\mapsto g^{-1}hg=\sum_{i}(g^{-1}h_{i}g)\otimes\La^{i}$.
\end{lem}

Let us determine the characters of the universal covering group $\tilde{J}_{r}(p)$
of $J_{r}(p)$. Since $J_{r}(p)\simeq\GL r\ltimes\jro(p)$, it is
sufficient to determine characters of the universal covering group
$\widetilde{\Gl}(r)$ of $\GL r$ and of $\jro(p)$ which come from
those of $\tilde{J}_{r}(p)$ by restriction. 

The characters of $\widetilde{\Gl}(r)$ is given as follows (Lemma
2.2 of \cite{kimura-2}). 
\begin{lem}
\label{lem:radon-conf-1-1} Any character $f:\widetilde{\Gl}(r)\to\cbatu$
is given by $f(x)=(\det x)^{a}$ for some $a\in\C$.
\end{lem}

Let us give the characters of $\jro(p)$. Let $\jroo(p)$ be the Lie
algebra of $\jro(p)$:
\[
\jroo(p)=\{X=\sum_{1\leq i<p}X_{i}w^{i}\mid X_{i}\in R\},
\]
where the Lie bracket of $X,Y\in\jroo(p)$ is given by $[X,Y]=\sum_{2\leq k<p}\sum_{i+j=k}[X_{i},Y_{j}]w^{k}$.
Note that we have an isomorphism $\jroo(p)\simeq\bigoplus_{1\leq i<p}Rw^{i}\simeq\oplus^{p-1}R$
as a vector space. To obtain a character of $\jro(p)$, we lift the
character of $\jroo(p)$ to that of $\jro(p)$ by the exponential
map. Since $\jro(p)$ is a simply connected Lie group, the exponential
map
\[
\exp:\jroo(p)\to\jro(p),\,X\mapsto\exp(X)=\sum_{k=0}^{\infty}\frac{1}{k!}X^{k}=\sum_{0\leq k<p}\frac{1}{k!}X^{k}
\]
is a biholomorphic map. Hence we can consider the inverse map $\log:\jro(p)\to\jroo(p)$,
which, for $h=1_{r}+\sum_{1\leq i<p}h_{i}w^{i}\in\jro(p)$, defines
$\te_{k}(h)\in R$ by
\begin{align}
\log h & =\log\left(1_{r}+h_{1}w+\cdots+h_{p-1}w^{p-1}\right)\nonumber \\
 & =\sum_{1\leq k<p}\frac{(-1)^{k+1}}{k}\left(h_{1}w+\cdots+h_{p-1}w^{p-1}\right)^{k}\nonumber \\
 & =\sum_{1\leq k<p}\theta_{k}(h)w^{k}.\label{eq:char-conf-0}
\end{align}
Here $\te_{k}(h)$ is given as a sum of monomials of noncommutative
elements $h_{1},\dots,h_{p-1}\in R$. If a weight of $h_{i}$ is defined
to be $i$, then the monomials appearing in $\te_{k}(h)$ has the
weight $k$. For example we have 
\begin{align*}
\te_{1}(h) & =h_{1},\\
\te_{2}(h) & =h_{2}-\frac{1}{2}h_{1}^{2},\\
\te_{3}(h) & =h_{3}-\frac{1}{2}(h_{1}h_{2}+h_{2}h_{1})+\frac{1}{3}h_{1}^{3},\\
\te_{4}(h) & =h_{4}-\frac{1}{2}(h_{1}h_{3}+h_{2}^{2}+h_{3}h_{1})+\frac{1}{3}(h_{1}^{2}h_{2}+h_{1}h_{2}h_{1}+h_{2}h_{1}^{2})-\frac{1}{4}h_{1}^{4}.
\end{align*}

\begin{lem}
\label{lem:Radon-conf-2-1}Let $\chi:\jro(p)\to\cbatu$ be a character
obtained from that of $\tilde{J}_{r}(p)$ by restricting it to $\jro(p)$.
Then there exists $\al=(\al_{1},\dots,\al_{p-1})\in\C^{p-1}$ such
that 
\begin{equation}
\chi(h;\al)=\exp\left(\sum_{1\leq i<p}\al_{i}\Tr\,\theta_{i}(h)\right).\label{eq:char-conf-1}
\end{equation}
Conversely, $\chi$ defined by (\ref{eq:char-conf-1}) gives a character
of $\jro(p)$.
\end{lem}

\begin{proof}
See Lemma 2.7 of \cite{kimura-2}.
\end{proof}
By virtue of the isomorphism in Lemma \ref{lem:radon-conf-1}, we
determine the characters of $\tilde{J}_{r}(p)$ as a consequence of
Lemmas \ref{lem:radon-conf-1-1} and \ref{lem:Radon-conf-2-1}.
\begin{prop}
\label{prop:Radon-conf-3-1}Any character $\chi_{p}:\tilde{J}_{r}(p)\to\cbatu$
is given by 
\[
\chi_{p}(h;\al)=(\det h_{0})^{\al_{0}}\exp\left(\sum_{1\leq i<p}\al_{i}\Tr\,\theta_{i}(\underline{h})\right),
\]
for some $\al=(\al_{0},\al_{1},\dots,a_{p-1})\in\C^{p}$, where $\underline{h}\in\jro(p)$
is defined by $h=\sum_{0\leq i<p}h_{i}w^{i}=\sum_{0\leq i<p}h_{0}(h_{0}^{-1}h_{i})w^{i}=h_{0}\cdot\underline{h}$.
\end{prop}

Now the characters of the group $\tilde{H}_{\lm}$ are given as follows.
\begin{prop}
\label{prop:char-conf}For a character $\chi_{\lm}:\tH_{\lm}\to\cbatu$,
there exists $\alpha=(\alpha^{(1)},\dots,\alpha^{(\ell)})\in\C^{n}$,
$\alpha^{(k)}=(\alpha_{0}^{(k)},\alpha_{1}^{(k)},\dots,\alpha_{n_{k}-1}^{(k)})\in\C^{n_{k}}$
such that 
\[
\chi_{\lm}(h;\al)=\prod_{1\leq k\leq\ell}\chi_{n_{k}}(h^{(k)};\al^{(k)}),\quad h=(h^{(1)},\cdots,h^{(\ell)})\in\tilde{H}_{\lm},\;h^{(k)}\in\tilde{J}_{r}(n_{k}).
\]
\end{prop}

\begin{cor}
In the case $\lm=(1,\dots,1)$, A character $\chi:=\chi_{\lm}$ has
the form 
\[
\chi(h;\al)=\prod_{1\leq k\leq n}(\det h^{(k)})^{\al^{(k)}},\quad h=\diag(h^{(1)},\dots,h^{(n)}),\quad h^{(k)}\in\widetilde{\Gl}(r)
\]
with $\al=(\al^{(1)},\dots,\al^{(n)})\in\C^{n}$. 
\end{cor}

\subsection{Definition of HGF of type $\protect\lm$}

Let a character $\chi_{\lm}:=\chi_{\lm}(\cdot;\al)$ of $\tilde{H}_{\lm}$
be given. To define the HGF as a Radon transform of the character
$\chi_{\lm},$ we prepare the space of independent variables of the
HGF. Let $m$ be an integer such that $r<m<N$ and define a Zariski
open subset $Z\subset\mat'(m,N)$ as follows, where $\mat'(m,N)=\{z\in\mat(m,N)\mid\rank\,z=m\}$.
According as the partition $\lm=(n_{1},\dots,n_{\ell})$ of $n$,
we write $z\in\mat'(m,N)$ as
\[
z=(z^{(1)},\dots,z^{(\ell)}),\;z^{(j)}=(z_{0}^{(j)},\dots,z_{n_{j}-1}^{(j)}),\;z_{k}^{(j)}\in\mat(m,r)
\]
and put 
\[
Z=\{z=(z^{(1)},\dots,z^{(\ell)})\in\mat'(m,N)\mid\rank z_{0}^{(k)}=r\;(1\leq k\leq\ell)\}.
\]
Also we take $T=\gras(r,m)=\GL r\setminus\mat'(r,m)$ as the space
of integration variables. We denote by $t=(t_{a,b})\in\mat'(r,m)$
the homogeneous coordinates of $T$ and by $[t]$ the point of $T$
with the homogeneous coordinate $t$.

For $z\in Z$, consider $\chi_{\lm}(tz;\al)$, where $tz=(tz^{(1)},\dots,tz^{(\ell)})\in\mat(r,N)$.
Here, $tz^{(j)}=(tz_{0}^{(j)},\dots,tz_{n_{j}-1}^{(j)})$ is identified
with $\sum_{0\leq k<n_{j}}tz_{k}^{(j)}\otimes\La^{k}\in J_{r}(n_{j})$
and $tz$ is identified with $\diag(tz^{(1)},\dots,tz^{(\ell)})\in H_{\lm}$.

Assume here that the character $\chi_{\lm}$ satisfies the following
condition.
\begin{assumption}
\label{assu:Radon-conf-4}(i) $\al_{0}^{(j)}\notin\Z$ for $1\leq j\leq\ell$,

(ii) $\al_{n_{j}-1}^{(j)}\neq0$ if $n_{j}\geq2$,

(iii) $\al_{0}^{(1)}+\cdots+\al_{0}^{(\ell)}=-m$.
\end{assumption}

By Assumption \ref{assu:Radon-conf-4} (iii), we see that $\chi_{\lm}(tz;\al)$
satisfies 
\begin{equation}
\chi_{\lm}((gt)z;\al)=(\det g)^{-m}\chi_{\lm}(tz;\al),\quad g\in\GL r,\label{eq:radon-1}
\end{equation}
which implies that $\chi_{\lm}(tz;\al)$ gives a multivalued analytic
section of the line bundle on $T$ associated with the character $\rho_{m}:\GL r\to\cbatu,\rho_{m}(g)=(\det g)^{m}$.
The branch locus of $\chi_{\lm}(tz;\al)$ on $T$ is 
\[
\bigcup_{1\leq j\leq\ell}S_{z}^{(j)},\quad S_{z}^{(j)}:=\{[t]\in T\mid\det(tz_{0}^{(j)})=0\}.
\]
Put $X_{z}:=T\setminus\cup_{1\leq j\leq\ell}S_{z}^{(j)}$, which is
a complement of the arrangement $\{S_{z}^{(1)},\dots,S_{z}^{(\ell)}\}$
of hypersurfaces of degree $r$ in $T$. 

We need $\tau(t)$, an $r(m-r)$-form in $t$-space, which can be
given as follows. For the homogeneous coordinates $t$ of $T$, put
$t=(t',t'')$ with $t'\in\mat(r),t''\in\mat(r,m-r)$ and consider
the affine neighbourhood $U=\{[t]\in T\mid\det t'\neq0\}$. Then we
can take affine coordinates $u$ of $U$ defined by $u=(t')^{-1}t''$.
Put $du:=\wedge_{i,j}du_{i,j}$, then we give $\tau(t)$ by
\begin{equation}
\tau(t)=(\det t')^{m}du.\label{eq:radon-2}
\end{equation}

\begin{example*}
In the case $T=\gras(1,m)=\Ps^{m-1}$ with the homogeneous coordinates
$t=(t_{1},\dots,t_{m})$, we take $\tau(t)=\sum_{1\leq j\leq m}(-1)^{j+1}t_{j}dt_{1}\wedge\cdots\wedge\widehat{dt_{j}}\wedge\cdots\wedge dt_{m}$.
Then in the coordinate neighbourhood $U=\{[t]\in T\mid t_{1}\neq0\}$
with the affine coordinates $(u_{2},\dots,u_{m})=(t_{2}/t_{1},\dots,t_{m}/t_{1})$,
we have 
\[
\tau(t)=t_{1}^{m}d\left(\frac{t_{2}}{t_{1}}\right)\wedge\cdots\wedge d\left(\frac{t_{m}}{t_{1}}\right)=t_{1}^{m}du_{2}\wedge\cdots\wedge du_{m}.
\]
\end{example*}
For $\tau(t)$ given by (\ref{eq:radon-2}), we have 
\begin{equation}
\tau(gt)=(\det g)^{m}\tau(t),\quad g\in\GL r.\label{eq:radon-3}
\end{equation}
Then, by virtue of (\ref{eq:radon-1}) and (\ref{eq:radon-3}), we
see that $\chi_{\lm}(tz;\al)\cdot\tau(t)$ gives a multivalued $r(m-r)$-form
on $X_{z}$. 
\begin{defn}
For a character $\chi_{\lm}(\cdot;\al)$ of the group $\tH_{\lm}$
satisfying Assumption \ref{assu:Radon-conf-4}, 
\begin{equation}
F_{\lm}(z,\al;C)=\int_{C(z)}\chi_{\lm}(tz;\al)\cdot\tau(t)\label{eq:confl-1}
\end{equation}
is called the Radon HGF of type $\lm$. Here $C(z)$ is an $r(m-r)$-cycle
of the homology group of locally finite chains $H_{r(m-r)}^{\Phi_{z}}(X_{z};\cL_{z})$
of $X_{z}$ with coefficients in the local system $\cL_{z}$ and with
the family of supports $\Phi_{z}$ determined by $\chi_{\lm}(tz;\al)$. 
\end{defn}

We briefly explain about the homology group $H_{r(m-r)}^{\Phi_{z}}(X_{z};\cL_{z})$.
For the detailed explanation, we refer to \cite{kimura-2} and references
therein. Write the integrand of (\ref{eq:confl-1}) as 
\[
\chi_{\lm}(tz;\al)=f(t,z)\exp(g(t,z)),
\]
where
\[
f(t,z)=\prod_{j=1}^{\ell}\left(\det(tz_{0}^{(j)})\right)^{\al_{0}^{(j)}},\quad g(t,z)=\sum_{1\leq j\leq n}\sum_{1\leq k<n_{j}}\al_{k}^{(j)}\Tr\,\theta_{k}(\underline{tz}^{(k)}).
\]
Note that $f(t,z)\cdot\tau(t)$ concerns the multivalued nature of
the integrand whose ramification locus is $\cup_{j}S_{z}^{(j)}$.
On the other hand, $g(t,z)$ is a rational function on $T$ with a
pole divisor $\cup_{j;n_{j}\geq2}S_{z}^{(j)}$ and concerns the nature
of exponential increase to infinity or exponential decrease to zero
of the integrand when $[t]$ approaches to the pole divisor $\cup_{j;n_{j}\geq2}S_{z}^{(j)}$.
The monodromy of $f(t,z)\cdot\tau(t)$, which is the same as that
of $\chi_{\lm}(tz;\al)\cdot\tau(t)$, defines a rank one local system
$\cL_{z}$ on $X_{z}$. On the other hand, $g_{z}:=g|_{X_{z}}:X_{z}\to\C$
defines a family $\Phi_{z}$ of closed subsets of $X_{z}$ by the
condition
\[
A\in\Phi_{z}\iff A\cap g_{z}^{-1}(\{w\in\C\mid\mathrm{Re}\,w\geq a\})\ \ \mbox{is compact for any \ensuremath{a\in\R}}.
\]
Then $\Phi_{z}$ satisfies the condition of a family of supports \cite{kimura-2,Pham-1,Pham-2}
and we can consider a homology groups of locally finite chains with
coefficients in the local system $\cL_{z}$ whose supports belong
to $\Phi_{z}$. This homology group is denoted by $H_{\bullet}^{\Phi_{z}}(X_{z};\cL_{z})$.
Moreover there is a Zariski open subset $V\subset Z$ such that 

\[
\bigcup_{z\in V}H_{r(m-r)}^{\Phi_{z}}(X_{z};\cL_{z})\to V,
\]
which maps $H_{r(m-r)}^{\Phi_{z}}(X_{z};\cL_{z})$ to $z$, gives
a local system on $V$ \cite{kimura-2}. We take its local section
as $C=\{C(z)\}$ to obtain the Radon HGF of type $\lm$. 

We give an expression of $F_{\lm}$ in terms of the affine coordinates
$u=(u_{i,j})=(t')^{-1}t''$ of the chart $U=\{[t]\in T\mid\det t'\neq0\}$.
Using (\ref{eq:radon-1}) and (\ref{eq:radon-2}), we have 
\begin{align*}
F_{\lm}(z,\al;C) & =\int_{C(z)}\chi_{\lm}(\vec{u}z;\al)du\\
 & =\int_{C(z)}\prod_{j=1}^{\ell}\left(\det(\vec{u}z_{0}^{(j)})\right)^{\al_{0}^{(j)}}\exp\left(\sum_{1\leq j\leq n}\sum_{1\leq k<n_{j}}\al_{k}^{(j)}\Tr\,\theta_{k}(\underline{\vec{u}z}^{(k)})\right)du,
\end{align*}
where $\vec{u}=(1_{r},u)$. In the case $\lm=(1,\dots,1)$, the Radon
HGF is written as 
\[
F(z,\al;C)=\int_{C(z)}\prod_{j=1}^{n}\left(\det(tz^{(j)})\right)^{\al^{(j)}}\tau(t)=\int_{C(z)}\prod_{j=1}^{n}\left(\det(\vec{u}z^{(j)})\right)^{\al^{(j)}}du
\]
and is said to be of \emph{non-confluent type}. 

We give an important property for the Radon HGF which states the covariance
of the function under the action of $\GL m\times\hlam$ on $Z$. we
see that the action
\[
\GL m\times\mnm\times\hlam\ni(g,z,h)\mapsto gzh\in\mnm
\]
induces that on the set $Z$. The following is Proposition 2.12 of
\cite{kimura-2}.
\begin{prop}
\label{prop:covariance-1}For the Radon HGF of type $\lm$, we have
the formulae

(1) $F_{\lm}(gz,\al;C)=\det(g)^{-r}F_{\lm}(z,\al;\tilde{C}),\quad g\in\GL m,$

(2) $F_{\lm}(zh,\al;C)=F_{\lm}(z,\al;C)\chi_{\lm}(h;\al),\quad h\in\tilde{H}_{\lm}$.
\end{prop}

\section{Contiguity relations}

\subsection{Capelli identity and Cayley's formula}

Let $\mathcal{P}(\mat(r))$ be the ring of polynomials on the matrix
space $\mat(r)$ with the coordinates $(x_{i,j})$. Put $\pa_{i,j}:=\pa/\pa x_{i,j}$.
Let $E_{i,j}$ be the $(i,j)$-th matrix unit, namely a matrix unit
whose unique nonzero entry is the $(i,j)$-th entry. Let $E'_{i,j}$
be the left invariant vector field on $\mat(r)$ corresponding to
$E_{i,j}\in\gl(r)$:
\[
E'_{i,j}f:=\frac{d}{ds}f(x\,\expo(sE_{i,j}))|_{s=0},\qquad E'_{i,j}=\sum_{a=1}^{r}x_{a,i}\pa_{a,j}.
\]
 The following theorem is known as the Capelli identity. See p507
of \cite{Fulton-Harris}.
\begin{thm*}
(Capelli) We have the identity
\begin{equation}
\det(E'_{i,j}+(r-j)\de_{i,j})=\det(x_{i,j})\det(\pa_{i,j}),\label{eq:capelli-1}
\end{equation}
where the determinant in the left hand side is a ``column determinant'',
namely, for an $r\times r$ matrix $A=(a_{i,j})$ with entries in
a noncommutative ring, we define 
\[
\det A:=\sum_{\sm\in\mathfrak{S}_{r}}\sgn(\sm)a_{\sm(1),1}a_{\sm(2),2}\cdots a_{\sm(r),r}.
\]
\end{thm*}
From the Capelli identity, we can obtain an important formula called
Cayley's formula, which is obtained as an example of the theory of
prehomogeneous vector spaces. Since its proof is elementary, we give
it for the sake of completeness.
\begin{prop*}
(Cayley) For $f=\det(x_{i,j})$ we have 
\begin{equation}
\det(\pa_{i,j})f^{s}=s(s+1)\cdots(s+r-1)f^{s-1}.\label{eq:capelli-1-1}
\end{equation}
\end{prop*}
\begin{proof}
In the Capelli indentity (\ref{eq:capelli-1}), the operator in the
left hand side is denoted as $P$. Then 
\[
P=\sum_{\sm\in\mathfrak{S}_{r}}\sgn(\sm)(E'_{\sm(1),1}+(r-1)\de_{\sm(1),1})(E'_{\sm(2),2}+(r-2)\de_{\sm(2),2})\cdots E'_{\sm(r),r}
\]
 by definition. We apply $P$ to $f^{s}$. We assert the identity:
\begin{equation}
E'_{i,j}f^{s}=\begin{cases}
0 & i\neq j,\\
sf^{s} & i=j.
\end{cases}\label{eq:capelli-2}
\end{equation}
 In fact, since $E'_{i,j}f^{s}=s\,f^{s-1}\cdot E'_{i,j}f$, we compute
$E'_{i,j}f$ firstly in case $i\neq j$ . We have 
\[
E'_{i,j}f=(\sum_{a=1}^{r}x_{a,i}\pa_{a,j})\left|\begin{array}{ccccc}
 & x_{1,i} &  & x_{1,j}\\
\cdots & x_{2,i} & \cdots & x_{2,j} & \cdots\\
 & \vdots &  & \vdots\\
 & x_{r,i} &  & x_{r,j}
\end{array}\right|=\left|\begin{array}{ccccc}
 & x_{1,i} &  & x_{1,i}\\
\cdots & x_{2,i} & \cdots & x_{2,i} & \cdots\\
 & \vdots &  & \vdots\\
 & x_{r,i} &  & x_{r,i}
\end{array}\right|=0.
\]
In case $i=j$, similar computation shows that $E'_{i,i}f=f$. Hence
we obtain (\ref{eq:capelli-2}). Identity (\ref{eq:capelli-2}) says
that when $P$ is applied to $f^{s}$, the nonzero contribution comes
only from the term for $\sm=id$. Then we have 
\begin{align*}
Pf^{s} & =(E'_{1,1}+(r-1))(E'_{2,2}+(r-2))\cdots E'_{r,r}f^{s}\\
 & =(s+r-1)(s+r-2)\cdots(s+1)s\,f^{s},
\end{align*}
which is equal to $f\cdot\det(\pa_{ij})f^{s}$ by virtue of the Capelli
identity. Hence we obtain Cayley's formula (\ref{eq:capelli-1-1}).
\end{proof}
\begin{rem}
\label{rem:Cont-1} In case $r=1$, Cayley's formula reduces to an
obvious identity
\[
\frac{d}{dx}x^{s}=sx^{s-1}.
\]
Construction of the contiguity relations for Gelfand's HGF is based
on this simple identity. In the context of the theory of prehomogeneous
vector space, Cayley's formula can be understood as follows. A prehomogeneous
vector space is a triple $(G,\rho,V)$ of a reductive algebraic group
$G$ over $\C$, and a complex vector space $V$ on which $G$ acts
through an algebraic action $\rho$ prehomogeneously, namely, there
is a Zariski open orbit. In our case $(G,\rho,V)=(\GL r,\rho,\mat(r))$
with the action $\rho$ defined by matrix multiplication 
\[
\rho:\GL r\to\GL V,\quad\rho(g)\cdot x=gx,\quad x\in V.
\]
 The open orbit is $O=\{x\in\mat(r)\mid\det x\neq0\}$ and its complement
$S=V\setminus O$ is defined as the zero of $f=\det x$, where $f$
is called the relative invariant of the prehomogeneous vector space.
Cayley's formula gives an analytic continuation of the complex power
$f^{s}$ with respect to $s$. In this context $\bb(s):=s(s+1)\cdots(s+r-1)$
is called the $b$-function.
\end{rem}

\subsection{Contiguity relation for Radon HGF}

As in Section \ref{sec:Radon-HGF}, let $r,m,N$ be integers such
that $1\leq r<m<N$ and $N=nr$ for a positive integer $n$. 

\subsubsection{Non-confluent case}

The Radon HGF of non-confluent type corresponds to the partition $\lm=(1,\dots,1)$
of $n$. For this partition, we associate the subgroup
\[
H:=H_{(1,\dots,1)}=\left\{ A=\left(\begin{array}{ccc}
h^{(1)}\\
 & \ddots\\
 &  & h^{(n)}
\end{array}\right)\mid h^{(k)}\in\GL r\quad(1\leq k\le n)\right\} 
\]
of $G=\GL N$. Let $\gee$ and $\ha$ denote the Lie algebras of $G$
and $H$, respectively. Consider the center $\fa=\{A\in\ha\mid[A,X]=0\;\text{for}\;\forall X\in\ha\}$
of $\ha$. It is easily seen that $\fa$ is given by 
\[
\fa=\left\{ A=\left(\begin{array}{ccc}
a_{1}1_{r}\\
 & \ddots\\
 &  & a_{n}1_{r}
\end{array}\right)\mid a_{k}\in\C\quad(1\leq k\le n)\right\} ,
\]
and that $\ha$ is recovered as the centralizer of $\fa$ in $\gee$.
We consider simultaneous eigenspace decomposition of $\gee$ with
respect to the commuting family $\{\ad\,A\mid A\in\fa\}$: 
\begin{align*}
\gee & =\ha\oplus\bigoplus_{i\neq j}\gee_{\ep^{(i)}-\ep^{(j)}},\\
\gee_{\ep^{(i)}-\ep^{(j)}} & =\{X\in\gee\mid[A,X]=(\ep^{(i)}(A)-\ep^{(j)}(A))X\quad\forall A\in\fa\},
\end{align*}
where $\ep^{(i)}\in\fa^{*}$ is defined by $\ep^{(i)}(A)=a_{i}$ for
$A=\diag(a_{1}1_{r},\dots,a_{n}1_{r})\in\fa$. The eigenspaces are
described as follows. Let $X\in\gee$ be written in the form of block
matrix $X=(X^{(i,j)})_{1\leq i,j\leq n}$ , where $X^{(i,j)}\in\mat(r)$
is the $(i,j)$-th block. Then 
\begin{align*}
\gee_{\ep^{(i)}-\ep^{(j)}} & =\{X\in\gee\mid X^{(i',j')}=0\;\text{for any}\;(i',j')\neq(i,j)\}\\
 & =\bigoplus_{1\leq p,q\leq r}\C\cdot E_{p,q}^{(i,j)},
\end{align*}
where $E_{p,q}^{(i,j)}\in\gee$ is the matrix unit whose only nozero
entry locates at the $(p,q)$-th entry of the $(i,j)$-th block. Let
us give the contiguity operators in this case. For $X\in\gee$, define
the first order differential operator on $\mat'(m,N)$ by
\begin{equation}
(L_{X}f)(z):=\frac{d}{ds}f(z\cdot\exp sX)|_{s=0}.\label{eq:cont-non-0}
\end{equation}
 In particular, for $X=E_{p,q}^{(i,j)}$, $L_{X}$ is denoted as $L_{p,q}^{(i,j)}$.
Then we define the differetial operator of order $r$ by
\[
\cL^{(i,j)}:=\det(L_{p,q}^{(i,j)})_{1\leq p,q\leq r},
\]
 where the determinant is the column determinant.

Let us give the explicit form of the operator $\cL^{(i,j)}$. Taking
account that we are considering the case $\lm=(1,\dots,1)$, put 
\begin{equation}
z=(z^{(1)},\dots,z^{(n)})\in\mat'(m,N),\quad z^{(j)}=(z_{1}^{(j)},\dots,z_{r}^{(j)})=(z_{a,b}^{(j)})_{0\leq a<m,1\leq b\leq r}\in\mat'(m,r).\label{eq:cont-non-0-1}
\end{equation}
Accordingly we put 
\begin{equation}
\pa^{(j)}=(\pa_{1}^{(j)},\dots,\pa_{r}^{(j)}),\quad\pa_{q}^{(j)}=(\pa_{a,q}^{(j)})_{0\leq a<m},\quad\pa_{a,q}^{(j)}:=\frac{\pa}{\pa z_{a,q}^{(j)}}.\label{eq:cont-non-0-2}
\end{equation}

\begin{lem}
\label{lem:Cont-2}We have 
\begin{align}
L_{p,q}^{(i,j)} & =\tr z_{p}^{(i)}\pa_{q}^{(j)}=\sum_{0\leq a<m}z_{a,p}^{(i)}\pa_{a,q}^{(j)},\label{eq:cont-non-1}\\
\cL^{(i,j)} & =\det(\,^{t}z^{(i)}\pa^{(j)}).\label{eq:cont-non-2}
\end{align}
\end{lem}

\begin{proof}
For the definition (\ref{eq:cont-non-0}) of the operator $L_{p,q}^{(i,j)}$,
we consider the action of the $1$-parameter subgroup $s\mapsto\exp\left(sE_{p,q}^{(i,j)}\right)$
on $z$: 
\begin{align*}
z\,\exp(sE_{pq}^{(i,j)}) & =(z^{(1)},\dots,z^{(n)})(1+sE_{pq}^{(i,j)}+O(s^{2}))\\
 & =(z^{(1)},\dots,z^{(i)},\dots,z^{(j)}+sz^{(i)}E_{pq}^{(i,j)},\dots,z^{(n)})+O(s^{2}),
\end{align*}
where $E_{pq}^{(i,j)}$ in the second line is identified with a matrix
unit in $\mat(r)$ whose nonzero entry locates at $(p,q)$-th entry.
Since
\[
z^{(j)}+sz^{(i)}E_{pq}^{(i,j)}=(z_{1}^{(j)},\dots,z_{p}^{(j)},\dots,z_{q}^{(j)}+sz_{p}^{(i)},\dots,z_{r}^{(j)}).
\]
It follows that for a function $f$ of $z$ we have
\[
L_{p,q}^{(i,j)}f=\frac{d}{ds}f(z\,\exp sE_{p,q}^{(i,j)})|_{s=0}=\sum_{0\leq a<m}z_{a,p}^{(i)}\frac{\pa f}{\pa z_{a,q}^{(j)}}=(\tr z_{p}^{(i)}\pa_{q}^{(j)})f.
\]
Hence we have $(L_{p,q}^{(i,j)})_{1\leq p,q\leq r}=\,^{t}z^{(i)}\pa^{(j)}$
and $\cL^{(i,j)}=\det(L_{p,q}^{(i,j)})_{1\leq p,q\leq r}=\det(\,^{t}z^{(i)}\pa^{(j)})$. 
\end{proof}
\begin{rem}
In the case $i\neq j$, the operators $L_{p,q}^{(i,j)}\;(1\leq p,q\leq r)$
commute each other. Hence the determinant $\det(L_{p,q}^{(i,j)})_{1\leq p,q\leq r}$
is defined as column determinant or as row determinant. 
\end{rem}

According as the partition $\lm=(n_{1},\dots,n_{\ell})$ of $n$,
we write $z\in\mat'(m,N)$ as
\[
z=(z^{(1)},\dots,z^{(\ell)}),\;z^{(j)}=(z_{0}^{(j)},\dots,z_{n_{j}-1}^{(j)}),\;z_{k}^{(j)}\in\mat(m,r)
\]

The following theorem is one of our main result.
\begin{thm}
\label{thm:conti-nonconf}The contiguity relations for the Radon HGF
of non-confluent type are given by

\begin{equation}
\cL^{(i,j)}F(z,\al)=\bb(\al^{(j)})F(z,\al+\ep^{(i)}-\ep^{(j)}),\quad1\leq i\neq j\leq n,\label{eq:cont-nonconf-1}
\end{equation}
where $\cL^{(i,j)}=\det(\,^{t}z^{(i)}\pa^{(j)})$, which is a differential
operator of order $r$, and $\bb(s)=s(s+1)\cdots(s+r-1)$, which is
called the $b$-function. Moreover $\al\mapsto$ $\al+\ep^{(i)}-\ep^{(j)}$
for $\al=(\al^{(1)},\dots,\al^{(n)})$ implies the change $\al^{(i)}\mapsto\al^{(i)}+1,\al^{(j)}\mapsto\al^{(j)}-1$
and $\al^{(k)}\mapsto\al^{(k)}\;(k\neq i,j)$.
\end{thm}

The proof of this theorem is given in Section \ref{subsec:Proof-nonconf}.

\subsubsection{Confluent case of type $\protect\lm$}

Let us give the contiguity relations for the Radon HGF of type $\lm=(n_{1},\dots,n_{\ell})$
when $n_{1}\geq2$. Let $G=\GL N$ and $\hlam$ be the group given
in Section \ref{subsec:Char-conf-1}. Let $\gee=\gl(N)$ and $\ha_{\lm}$
be the Lie algebra of $G$ and $H_{\lm}$, respectively. Then $\ha_{\lm}=\fj_{r}(n_{1})\oplus\cdots\oplus\fj_{r}(n_{\ell})$
with $\fj_{r}(p)=\{X=\sum_{0\leq i<p}X_{i}\otimes\La^{i}\mid X_{i}\in\mat(r)\}$.
Moreover let $\fa_{\lm}=\{A\in\ha_{\lm}\mid[A,X]=0\;\text{for}\;\forall X\in\ha_{\lm}\}$
be the center of $\ha_{\lm}$. It is easily seen that the Lie algebra
$\fa_{\lm}$ is explicitly given by
\begin{equation}
\fa_{\lm}=\left\{ A=A^{(1)}\oplus\cdots\oplus A^{(\ell)}\mid A^{(j)}=\sum_{0\leq k<n_{j}}(a_{k}^{(j)})1_{r}\otimes\La^{k}\in\fj_{r}(n_{j}),a_{k}^{(j)}\in\C\right\} ,\label{eq:cont-confl-1}
\end{equation}
and $\ha_{\lm}$ is recovered as the centralizer of $\fa_{\lm}$ in
$\gee$. Consider the generalized root space decomposition of $\gee$
with respect to the commuting family of Lie algebra homomphisms $\{\ad\,A\mid A\in\fa_{\lm}\}$.
It is easily seen that this decomposition is given by
\begin{align*}
\gee & =\ha_{\lm}\oplus\bigoplus_{i\neq j}\tilde{\gee}_{\ep^{(i)}-\ep^{(j)}},\\
\tilde{\gee}_{\ep^{(i)}-\ep^{(j)}} & =\{X\in\gee\mid\left(\ad\;A-(\ep^{(i)}(A)-\ep^{(j)}(A)\right)^{k}X=0\quad\forall A\in\fa_{\lm},\exists k\}.
\end{align*}
where $\ep^{(i)}\in\fa_{\lm}^{*}$ is defined by $\ep^{(i)}(A)=a_{0}^{(i)}$
for $A$ given in (\ref{eq:cont-confl-1}). Let us describe root vectors
in the generalized root space decomposition. As in the non-confluent
case, we express $X\in\gee$ in blockwise as $X=\left(X^{(i,j)}\right)_{1\leq i,j\leq\ell}$
with $X^{(i,j)}\in\mat(n_{i}r,n_{j}r)$ and moreover the $(i,j)$-th
block $X^{(i,j)}$ is also written in the form of block matrix as
\[
X^{(i,j)}=\left(X^{(i,j;p,q)}\right)_{0\leq p<n_{i},0\leq q<n_{j}},\quad X^{(i,j;p,q)}\in\mat(r).
\]
Then the generalized root space is 
\[
\tilde{\gee}_{\ep^{(i)}-\ep^{(j)}}=\{X\in\gee\mid X^{(i',j')}=0\quad\text{for}\quad\forall(i',j')\neq(i,j)\}
\]
 and the root space $\gee_{\ep^{(i)}-\ep^{(j)}}$ with respect to
$\{\ad\,A\mid A\in\fa_{\lm}\}$, which is defined by
\[
\gee_{\ep^{(i)}-\ep^{(j)}}=\{X\in\gee\mid\left(\ad\;A-(\ep^{(i)}(A)-\ep^{(j)}(A)\right)X=0\quad\forall A\in\fa_{\lm}\}
\]
for the root $\ep^{(i)}-\ep^{(j)}$, consists of $X\in\tilde{\gee}_{e_{i}-e_{j}}$
whose only nonzero $r\times r$ block is at the top and right end
corner of the block matrix $X^{(i,j)}=\left(X^{(i,j;p,q)}\right)_{0\leq p<n_{i},0\leq q<n_{j}}$,
namely
\[
\gee_{\ep^{(i)}-\ep^{(j)}}=\{X\in\gee\mid X^{(i',j';p',q')}=0\quad\text{for}\quad\forall(i',j',p',q')\neq(i,j,0,n_{j}-1)\}.
\]

\begin{example*}
Consider the case $\lm=(2,2),\;N=4r$. The roots are $\pm(\ep^{(1)}-\ep^{(2)})$
and the corresponding root spaces are 
\[
\gee_{\ep^{(1)}-\ep^{(2)}}=\left\{ \left[\begin{array}{cccc}
. & . & . & \bullet\\
. & . & . & .\\
. & . & . & .\\
. & . & . & .
\end{array}\right]\right\} ,\quad\gee_{\ep^{(2)}-\ep^{(1)}}=\left\{ \left[\begin{array}{cccc}
. & . & . & .\\
. & . & . & .\\
. & \bullet & . & .\\
. & . & . & .
\end{array}\right]\right\} ,
\]
where each dot $\cdot$ represents a block of $r\times r$ zero matrix
and $\bullet$ denotes a block with an arbitrary $r\times r$ matrix. 
\end{example*}
We denote by $E_{a,b}^{(i,j)}\;(1\leq a,b\leq r)$ the element of
$\gee_{\ep^{(i)}-\ep^{(j)}}$ which is a matrix unit whose nonzero
entry locates at the $(a,b)$-th entry of the unique nonzero block
of $r\times r$ matrix. We see that $\{E_{a,b}^{(i,j)}\}_{1\leq a,b\leq r}$
forms a basis of $\gee_{\ep^{(i)}-\ep^{(j)}}$. As in the case of
non-confluent HGF, we define for $X\in\gee$ the differential operator
on $\mat'(m,N)$ by 
\[
(L_{X}f)(z):=\frac{d}{ds}f(z\,\exp sX)|_{s=0}.
\]
 In particular, the operator $L_{X}$ for the root vector $X=E_{a,b}^{(i,j)}$
will be denoted as $L_{a,b}^{(i,j)}$. Then the differential operator
$\cL^{(i,j)}:=\det\left(L_{a,b}^{(i,j)}\right)$ of order $r$ will
serve as a contiguity operator. Note that the differential operators
$L_{a,b}^{(i,j)},1\leq a,b\leq r$, are commuting operators, so the
determinant is considered as a column determinant as well as a row
determinant. These two determinants coincide. 

Let us give the explicit form of the operator $\cL^{(i,j)}$. According
as the partition $\lm=(n_{1},\dots,n_{\ell})$ of $n$, we write $z\in\mat'(m,N)$
as
\[
z=(z^{(1)},\dots,z^{(\ell)})\in\mat'(m,N),\quad z^{(j)}=(z_{0}^{(j)},\dots,z_{n_{j}-1}^{(j)}),\;z_{k}^{(j)}=(z_{k;a,b}^{(j)})_{0\leq a<m,1\leq b\leq r}\in\mat(m,r).
\]
Accordingly we put 
\[
\pa^{(j)}=(\pa_{0}^{(j)},\dots,\pa_{n_{j}-1}^{(j)}),\quad\pa_{k}^{(j)}=(\pa_{k;a,b}^{(j)})_{0\leq a<m,1\leq b\leq r},\quad\pa_{k;a,b}^{(j)}:=\frac{\pa}{\pa z_{k;a,b}^{(j)}}.
\]
 It should be careful that the way of indexation here is different
from those used in (\ref{eq:cont-non-0-1}) and (\ref{eq:cont-non-0-2}). 

In a similar way as Lemma \ref{lem:Cont-2}, we can easily show the
following. 
\begin{lem}
We have 
\begin{align}
L_{a,b}^{(i,j)} & =\sum_{0\leq c<m}z_{0;c,a}^{(i)}\pa_{n_{j}-1;c,b}^{(j)},\label{eq:cont-conf-0-1}\\
\cL^{(i,j)} & =\det(\,^{t}z_{0}^{(i)}\pa_{n_{j}-1}^{(j)}).\label{eq:cont-conf-0-2}
\end{align}
\end{lem}

\begin{thm}
\label{thm:cont-conf} For the Radon HGF of type $\lm=(n_{1},\dots,n_{\ell})$,
we have the contiguity relations as follows.

(1) When $n_{j}=1$, we have 
\[
\cL^{(i,j)}F(z,\al)=\al_{0}^{(j)}(\al_{0}^{(j)}+1)\cdots(\al_{0}^{(j)}+r-1)\cdot F(z,\al+\ep^{(i)}-\ep^{(j)}).
\]

(2) When $n_{j}\geq2$, we have 
\begin{equation}
\cL^{(i,j)}F(z,\al)=(\al_{n_{j}-1}^{(j)})^{r}F(z,\al+\ep^{(i)}-\ep^{(j)}).\label{eq:cont-conf-2}
\end{equation}
Here $\cL^{(i,j)}=\det(\,^{t}z_{0}^{(i)}\pa_{n_{j}-1}^{(j)})$ and
$\al\mapsto$ $\al+\ep^{(i)}-\ep^{(j)}$ implies the change $\al^{(i)}\mapsto(\al_{0}^{(i)}+1,\al_{1}^{(i)},\dots,\al_{n_{i}-1}^{(i)})$
and $\al^{(j)}\mapsto(\al_{0}^{(j)}-1,\al_{1}^{(j)},\dots,\al_{n_{j}-1}^{(j)})$
with other $\al^{(k)}$ unchanged. 
\end{thm}

\subsection{\label{subsec:Proof-nonconf}Proof of Theorem \ref{thm:conti-nonconf}}

Since the HGF has the form
\[
F(z,\al;C)=\int_{C(z)}\chi(tz,\al)\cdot\tau=\int_{C(z)}\prod_{k=1}^{n}(\det tz^{(k)})^{\al_{k}}\cdot\tau
\]
and the differential operator $\cL^{(i,j)}=\det(L_{p,q}^{(i,j)})_{1\leq p,q\leq r}=\det(\,^{t}z^{(i)}\pa^{(j)})$
in (\ref{eq:cont-nonconf-1}) contains only the derivations $\pa_{a,b}^{(j)}=\pa/\pa z_{a,b}^{(j)}$
having the index $j$, we have
\[
\cL^{(i,j)}F(z,\al;C)=\int_{C(z)}\prod_{k\neq j}(\det tz^{(k)})^{\al_{k}}\cdot\cL^{(i,j)}(\det tz^{(j)})^{\al_{j}}\cdot\tau.
\]
 So we compute $\cL^{(i,j)}(\det tz^{(j)})^{\al_{j}}$. To avoid the
unnecessary cumbersome notation, we put $s:=\al_{j}$. It is enough
to show
\begin{equation}
\cL^{(i,j)}(\det tz^{(j)})^{s}=\bb(s)(\det tz^{(j)})^{s-1}(\det tz^{(i)}),\label{eq:cont-nonconf-2}
\end{equation}
where $\bb(s)=s(s+1)\cdots(s+r-1)$. We write the left and the right
hand side of (\ref{eq:cont-nonconf-2}) as $A(t,z)$ and $B(t,z)$,
respectively.
\begin{lem}
\label{lem:cont-key}For any $g\in\GL m$, there holds
\[
A(tg,z)=A(t,gz),\quad B(tg,z)=B(t,gz).
\]
 
\end{lem}

\begin{proof}
The second identity can be shown as 
\begin{align*}
B(tg,z) & =b(s)(\det(tg)z^{(j)})^{s-1}\cdot(\det(tg)z^{(i)})\\
 & =b(s)(\det t(gz^{(j)}))^{s-1}\cdot(\det t(gz^{(i)}))\\
 & =B(t,gz).
\end{align*}
To show the first identity, note that 
\begin{equation}
A(tg,z)=\det(\,^{t}z^{(i)}\pa^{(j)})\left(\det(tg)z^{(j)}\right)^{s}=\det(\,^{t}z^{(i)}\pa^{(j)})\left(\det t(gz^{(j)})\right)^{s}.\label{eq:cont-nonconf-3}
\end{equation}
 For $A(t,gz)$, we make a change of variable $z\to w$ defined by
$w=gz$. In particular $w^{(i)}=gz^{(i)},w^{(j)}=gz^{(j)}$. Then
to check the first identity it is sufficient to show 
\begin{equation}
\det(\,^{t}z^{(i)}\pa_{z}^{(j)})=\det(\,^{t}w^{(i)}\pa_{w}^{(j)}).\label{eq:cont-nonconf-4}
\end{equation}
 In fact, from (\ref{eq:cont-nonconf-3}) we have 
\[
A(tg,z)=\det(\,^{t}w^{(i)}\pa_{w}^{(j)})\left(\det tw^{(j)}\right)^{s}=A(t,w)=A(t,gz).
\]
 So, we show (\ref{eq:cont-nonconf-4}). Note that if one perform
a transformation $z\mapsto w:=gz,$ in particular $z^{(i)}\mapsto w^{(i)}=gz^{(i)}$
and $z^{(j)}\mapsto w^{(j)}=gz^{(j)}$, the derivation $\pa_{z}^{(j)}$
transforms to $\pa_{w}^{(j)}$ contravariantly to $z^{(j)}\mapsto w^{(j)}$,
namely, $\pa_{w}^{(j)}=\tr g^{-1}\pa_{z}^{(j)}$. It follows that
\[
\det(\,^{t}w^{(i)}\pa_{w}^{(j)})=\det\left(\tr(gz^{(i)})\tr g^{-1}\pa_{z}^{(j)}\right)=\det(\,^{t}z^{(i)}\pa_{z}^{(j)}).
\]
This shows (\ref{eq:cont-nonconf-4}).
\end{proof}
By the help of Lemma \ref{lem:cont-key}, we can reduce the proof
of (\ref{eq:cont-nonconf-2}) to that for a particularly chosen $t\in\mat'(r,m)$.
In fact, for $t\in\mat'(r,m)$, take $g\in\GL m$ and define $u\in\mat'(r,m)$
by $t=ug$, moreover put $z'=gz$, then Lemma \ref{lem:cont-key}
says that 
\[
A(u,z')=A(ug,g^{-1}z')=A(t,z),\;B(u,z')=B(ug,g^{-1}z')=B(t,z).
\]
So if one can show $A(u,z)=B(u,z)$ for a particularly chosen $u$
and for any $z\in\mat'(m,N)$, we are done. We can take $g$ so that
$u=tg^{-1}=(1_{r},0)$. Then it is sufficient to show the identity
(\ref{eq:cont-nonconf-2}) assuming $t=(1_{r},0)$. Then 
\[
tz^{(i)}=(1_{r},0)\left(\begin{array}{c}
{z_{0}^{(i)}}'\\
\vdots\\
{z_{r-1}^{(i)}}'\\
\\
{z_{m-1}^{(i)}}'
\end{array}\right)=\left(\begin{array}{c}
{z_{0}^{(i)}}'\\
\vdots\\
{z_{r-1}^{(i)}}'
\end{array}\right)=:\tilde{z}^{(i)}\in\mat(r),
\]
where ${z_{k}^{(i)}}'$ is the $k$-th row vector of the matrix $z^{(i)}\in\mat(m,r)$.
Similarly we have $tz^{(j)}=:\tilde{z}^{(j)}\in\mat(r)$ and the identity
(\ref{eq:cont-nonconf-2}) reduces to 
\begin{equation}
(\det\tilde{z}^{(i)})^{-1}\det(\,^{t}z^{(i)}\pa^{(j)})(\det\tilde{z}^{(j)})^{s}=b(s)(\det\tilde{z}^{(j)})^{s-1}.\label{eq:cont-non-5}
\end{equation}
To show this identity, we computes its left hand side. Note that 
\[
(\,^{t}\tilde{z}^{(i)})^{-1}\,^{t}z^{(i)}=\,^{t}(z^{(i)}\tilde{z}^{(i)-1})=(1_{r},y),\quad\text{for some }y\in\mat(r,m-r).
\]
Then the left hand side of (\ref{eq:cont-non-5}) is written as
\begin{align*}
\det((\,^{t}\tilde{z}^{(i)})^{-1}\,^{t}z^{(i)}\pa^{(j)})(\det\tilde{z}^{(j)})^{s} & =\det\left((1_{r},y)\left(\begin{array}{c}
\pa_{1}^{(j)}\\
\pa_{2}^{(j)}
\end{array}\right)\right)(\det\tilde{z}^{(j)})^{s}\\
 & =\det(\pa_{1}^{(j)}+y\pa_{2}^{(j)})(\det\tilde{z}^{(j)})^{s}.
\end{align*}
Here $\pa_{1}^{(j)}$ is the matrix obtained by choosing the first
$r$ rows from $\pa^{(j)}$ and $\pa_{2}^{(j)}$ is the matrix obtained
by taking the remained rows from $\pa^{(j)}$. Namely 
\begin{align*}
\tilde{z}^{(j)} & =\left(\begin{array}{ccc}
z_{01}^{(j)} & \dots & z_{0,r}^{(j)}\\
\vdots &  & \vdots\\
z_{r-1,1}^{(j)} & \dots & z_{r-1,r}^{(j)}
\end{array}\right),\\
\pa_{1}^{(j)} & =\left(\begin{array}{ccc}
\pa_{01}^{(j)} & \dots & \pa_{0,r}^{(j)}\\
\vdots &  & \vdots\\
\pa_{r-1,1}^{(j)} & \dots & \pa_{r-1,r}^{(j)}
\end{array}\right),\\
\pa_{2}^{(j)} & =\left(\begin{array}{ccc}
\pa_{r,1}^{(j)} & \dots & \pa_{r,r}^{(j)}\\
\vdots &  & \vdots\\
\pa_{m-1,1}^{(j)} & \dots & \pa_{m-1,r}^{(j)}
\end{array}\right).
\end{align*}
Note that $y$ depends only on the entries of $z^{(i)}$ and is annihilated
by any $\pa_{a,b}^{(j)}$ and that the terms in the operator $\det(\pa_{1}^{(j)}+y\pa_{2}^{(j)})$
containing some entry of $\pa_{2}^{(j)}$ annihilate $(\det\tilde{z}^{(j)})^{s}$.
Thus the identity to be shown is reduced to 
\[
\det(\pa_{1}^{(j)})(\det\tilde{z}^{(j)})^{s}=\bb(s)(\det\tilde{z}^{(j)})^{s-1}.
\]
 But this is nothing but Cayley's formula, and the proof of the theorem
is completed.

\subsection{\label{subsec:Proof-conf}Proof of Theorem \ref{thm:cont-conf}}

The proof of the assertion (1) of Theorem \ref{thm:cont-conf} is
carry out in a similar way as that for Theorem \ref{thm:conti-nonconf},
so we show the assertion (2). We apply the operator $\cL^{(i,j)}$
to 
\[
F(z,\al)=F(z,\al;C):=\int_{C(z)}\chi_{\lm}(tz,\al)\cdot\tau.
\]
 Recall that $n_{j}\geq2$, $z=(z^{(1)},\dots,z^{(\ell)})\in Z_{\lm}$,
$z^{(j)}=(z_{0}^{(j)},\dots,z_{n_{j}-1}^{(j)}),z_{k}^{(j)}\in\mat(m,r)$
and 
\begin{equation}
L_{a,b}^{(i,j)}=\sum_{0\leq c<m}(z_{0}^{(i)})_{ca}\frac{\pa}{\pa(z_{n_{j}-1}^{(j)})_{cb}}=\left(\tr z_{0}^{(i)}\pa_{n_{j}-1}^{(j)}\right)_{a,b}.\label{eq:proof-conf-1}
\end{equation}
The integrand has the form
\begin{align}
\chi_{\lm}(tz;\al) & =\prod_{k=1}^{\ell}\chi_{n_{k}}(tz^{(k)};\al^{(k)}),\nonumber \\
\chi_{n_{k}}(tz^{(k)};\al^{(k)}) & =(\det(tz_{0}^{(k)})^{\al_{0}^{(k)}}\prod_{1\leq p<n_{k}}\exp\left(\al_{p}^{(k)}\Tr\left(\te_{p}((tz_{0}^{(k)})^{-1}(tz^{(k)}))\right)\right).\label{eq:proof-conf-2}
\end{align}
Note that the operator $\cL^{(i,j)}=\det(\tr z_{0}^{(i)}\pa_{n_{j}-1}^{(j)})$
contains only the differentiations with respect the entries of the
matrix $z_{n_{j}-1}^{(j)}$. It follows that 
\[
\cL^{(i,j)}\cdot F(z,\al)=\int_{C(z)}\prod_{k\neq j}\chi_{n_{k}}(tz^{(k)};\al^{(k)})\cdot\left(\cL^{(i,j)}\cdot\chi_{n_{j}}(tz^{(j)};\al^{(j)})\right)\cdot\tau.
\]
From (\ref{eq:proof-conf-2}), the terms relating $z_{n_{j}-1}^{(j)}$
in $\chi_{n_{j}}(tz^{(j)};\al^{(j)})$ is $\te_{n_{j}-1}((tz_{0}^{(j)})^{-1}(tz^{(j)}))$.
Moreover, from the definition of $\te_{n_{j}-1}$ given in (\ref{eq:char-conf-0}),
we see that 
\[
\te_{n_{j}-1}((tz_{0}^{(j)})^{-1}(tz^{(j)}))=(tz_{0}^{(j)})^{-1}(tz_{n_{j}-1}^{(j)})+\left(\text{terms containing }tz_{k}^{(j)}\;(0\leq k<n_{j}-1)\right).
\]
 Then 
\begin{align*}
\cL^{(i,j)}\cdot\chi_{n_{j}}(tz^{(j)};\al^{(j)}) & =(\det(tz_{0}^{(j)})^{\al_{0}^{(j)}}\prod_{1\leq p<n_{j}-1}\exp\left(\al_{p}^{(j)}\Tr\left(\te_{p}((tz_{0}^{(j)})^{-1}(tz^{(j)}))\right)\right)\\
 & \qquad\times\cL^{(i,j)}\cdot\exp\left(\al_{n_{j}-1}^{(j)}\Tr\left(\te_{n_{j}-1}((tz_{0}^{(j)})^{-1}(tz^{(j)}))\right)\right).
\end{align*}
 So we compute $\cL^{(i,j)}\cdot\exp\left(\al_{n_{j}-1}^{(j)}\Tr\left((tz_{0}^{(j)})^{-1}(tz_{n_{j}-1}^{(j)})\right)\right)$.
Write $\cL^{(i,j)}$ as
\[
\cL^{(i,j)}=\det\left(L_{a,b}^{(i,j)}\right)_{1\leq a,b\leq r}=\sum_{\sm\in\Si_{r}}\sgn(\sm)L_{\sm(1),1}^{(i,j)}L_{\sm(2),2}^{(i,j)}\cdots L_{\sm(r),r}^{(i,j)},
\]
and put $X=(tz_{0}^{(j)})^{-1}$. Taking into account the form of
the operator $L_{a,b}^{(i,j)}$ given in (\ref{eq:proof-conf-1})
and $n_{j}\geq2$, we note that $L_{\sm(k),k}^{(i,j)}$ annihilates
$X$. Then 
\[
L_{\sm(k),k}^{(i,j)}\exp\left(\al_{n_{j}-1}^{(j)}\Tr\left(X(tz_{n_{j}-1}^{(j)})\right)\right)=\al_{n_{j}-1}^{(j)}\exp\left(\al_{n_{j}-1}^{(j)}\Tr\left(X(tz_{n_{j}-1}^{(j)})\right)\right)L_{\sm(k),k}^{(i,j)}\Tr\left(X(tz_{n_{j}-1}^{(j)})\right),
\]
and the last factor in the right hand side is 
\begin{align*}
L_{\sm(k),k}^{(i,j)}\Tr\left(X(tz_{n_{j}-1}^{(j)})\right) & =\Tr\left(X\cdot L_{\sm(k),k}^{(i,j)}(tz_{n_{j}-1}^{(j)})\right)\\
 & =\Tr\left(X\cdot(0,\dots,0,(tz_{0}^{(i)})_{\sm(k)},0,\dots,0)\right)\\
 & =\left(X(tz_{0}^{(i)})\right)_{k,\sm(k)},
\end{align*}
where $(0,\dots,0,(tz_{0}^{(i)})_{\sm(k)},0,\dots,0)\in\mat(r)$ represents
the matrix whose $k$-th column vector is the $\sm(k)$-th column
vector of $tz_{0}^{(i)}\in\mat(r)$ and the other columns are $0$.
Then
\begin{align*}
 & \cL^{(i,j)}\cdot\exp\left(\al_{n_{j}-1}^{(j)}\Tr\left(X(tz_{n_{j}-1}^{(j)})\right)\right)\\
 & =\left(\al_{n_{j}-1}^{(j)}\right)^{r}\exp\left(\al_{n_{j}-1}^{(j)}\Tr\left(X(tz_{n_{j}-1}^{(j)})\right)\right)\sum_{\sm\in\Si_{r}}\sgn(\sm)\prod_{1\leq k\leq r}\left(X(tz_{0}^{(i)})\right)_{k,\sm(k)}\\
 & =\left(\al_{n_{j}-1}^{(j)}\right)^{r}\exp\left(\al_{n_{j}-1}^{(j)}\Tr\left(X(tz_{n_{j}-1}^{(j)})\right)\right)\det\left(X(tz_{0}^{(i)})\right)\\
 & =\left(\al_{n_{j}-1}^{(j)}\right)^{r}\exp\left(\al_{n_{j}-1}^{(j)}\Tr\left(X(tz_{n_{j}-1}^{(j)})\right)\right)\det\left(tz_{0}^{(i)}\right)\det\left(tz_{0}^{(j)}\right)^{-1}.
\end{align*}
It follows that we have 
\begin{align*}
\cL^{(i,j)}\cdot\chi_{\lm}(tz,\al) & =\left(\al_{n_{j}-1}^{(j)}\right)^{r}\chi_{\lm}(tz;\al)\det\left(tz_{0}^{(i)}\right)\det\left(tz_{0}^{(j)}\right)^{-1}\\
 & =\left(\al_{n_{j}-1}^{(j)}\right)^{r}\chi_{\lm}(tz;\al+\ep^{(i)}-\ep^{(j)}).
\end{align*}
From this the contiguity relation (\ref{eq:cont-conf-2}) follows
and the proof of Theorem \ref{thm:cont-conf} is completed.

\section{\label{sec:Application-to-gamma}Application to gamma and beta functions
by Hermitian matrix integral}

\subsection{Beta and gamma functions by Hermitian matrix integral}

We recall here the beta and gamma functions defined by Hermitian matrix
integral. For the classical beta and gamma functions 
\begin{align*}
B(a,b) & :=\int_{0}^{1}u^{a-1}(1-u)^{b-1}du,\\
\G(a) & :=\int_{0}^{\infty}e^{-u}u^{a-1}du,
\end{align*}
their Hermitian matrix integral analogues are known. Let $\herm$
be the set of $r\times r$ complex Hermitian matrices. It is a real
vector space of dimension $r^{2}$. For $U=(U_{i,j})\in\herm,$ let
$dU$ be a volume form on $\herm$, which is the usual Euclidean volume
form given by

\begin{equation}
dU=\bigwedge_{i=1}^{r}dU_{i,i}\bigwedge_{i<j}\left(d\re(U_{i,j})\wedge d\im(U_{i,j})\right),\label{eq:hermInt-1}
\end{equation}
where $\bigwedge_{i=1}^{r}dU_{i,i}=dU_{1,1}\wedge dU_{2,2}\wedge\cdots\wedge dU_{r,r}$.
Note that this form can be written as 
\begin{equation}
dU=\left(\frac{\sqrt{-1}}{2}\right)^{r(r-1)/2}\bigwedge_{i=1}^{r}dU_{i,i}\bigwedge_{i\neq j}(dU_{i,j}\wedge dU_{j,i}).\label{eq:hermInt-2}
\end{equation}
Then the Hermitian matrix integral version of the beta and the gamma
are defined by 
\begin{align}
B_{r}(a,b) & :=\int_{0<U<1_{r}}(\det U)^{a-r}(\det(1_{r}-U))^{b-r}\,dU,\label{eq:herInt-3}\\
\G_{r}(a) & :=\int_{U>0}\etr(-U)(\det U)^{a-r}\,dU,\label{eq:hermInt-4}
\end{align}
respectively, where $\etr(U)=\exp(\Tr(U))$, $\Tr(U)$ being the trace
of $U$. The domain of integration is the set of positive definite
Hermitian matrices $U>0$ for the gamma, and the subset of $\herm$
satisfying $U>0$ and $1_{r}-U>0$ for the beta. It is known that
the gamma integral (\ref{eq:hermInt-4}) converges for $\re(a)>r-1$
and the beta integral (\ref{eq:herInt-3}) converges for $\re(a)>r-1,\re(b)>r-1$,
and they define holomorphic functions there. It is seen that $B_{r}$
and $\G_{r}$ reduce to the classical beta and gamma functions when
$r=1$, respectively. 

Next we consider the Radon HGF on the Grassmannian $\gras(2r,3r)$
of type $\lm=(1,1,1),(2,1)$ and explain that the above $B_{r}(a,b)$
and $\G_{r}(a)$ are understood as the Radon HGF of type $\lm=(1,1,1)$
and $(2,1)$, respectively. To establish a connection, we assume some
additional condition on the space of independent variable $z\in\mat'(2r,3r)$.
We state the condition for $z\in\mat'(2r,nr)$ with $n\geq3$ general. 

Let a partition $\lm=(n_{1},\dots,n_{\ell})$ of $n$ be given. We
say that $\mu=(m_{1},\dots,m_{\ell})\in\Z_{\geq0}^{\ell}$ is a subdiagram
of $\lm$ of weight $2$ if it satisfies 

\[
0\leq m_{k}\leq n_{k}\quad(\forall k)\quad\mbox{and}\quad|\mu|:=m_{1}+\cdots+m_{\ell}=2.
\]
More explicitly, $\mu$ with $|\mu|=2$ has the form either 
\begin{equation}
\mu=(0,\dots,0,\overset{i}{1},0,\dots,0,\overset{j}{1},0,\dots,0)\;\text{or }\mu=(0,\dots,0,\overset{i}{2},0,\dots,0).\label{eq:red-1}
\end{equation}
The first case means that $m_{i}=m_{j}=1$ and $m_{k}=0$ for $k\neq i,j$,
and the second case means that $m_{i}=2$ and $m_{k}=0$ for $k\neq i$.
Using this notation we define a Zariski open subset $Z_{\lm}\subset\mat'(2r,nr)$
of the space of independent variables of Radon HGF as follows. We
write $z\in\mat'(2r,nr)$ as $z=(z^{(1)},\dots,z^{(\ell)})$, where
$z^{(j)}$ is a $2r\times n_{j}r$ matrix which is called the $j$-th
block of $z$. For any $j$, $z^{(j)}$ is also written as
\[
z^{(j)}=(z_{0}^{(j)},\dots,z_{n_{j}-1}^{(j)}),\quad z_{p}^{(j)}\in\mat(2r,r).
\]
For a subdiagram $\mu\subset\lm$ with $|\mu|=2$, $\mu$ is either
of the form in (\ref{eq:red-1}). According as the form of $\mu$,
put
\[
z_{\mu}=(z_{0}^{(i)},z_{0}^{(j)})\;\text{or }z_{\mu}=(z_{0}^{(i)},z_{1}^{(i)}),
\]
respectively. Then $Z_{\lm}$ is defined as 
\[
Z_{\lm}:=\{z\in\mat(2r,nr)\mid\det z_{\mu}\neq0\;\text{for any subdiagram}\;\mu\subset\lm,|\mu|=2\}.
\]
It is easily seen that $Z_{\lm}$ is invariant by the action $\GL{2r}\curvearrowright\mat'(2r,nr)\curvearrowleft H_{\lm}$.
Taking into account the covariance property for the Radon HGF with
respect to the action of $\GL{2r}\times H_{\lm}$ given in Proposition
\ref{prop:covariance-1}, we take the independent variable $z$ to
a simpler form $\bx\in Z_{\lm}$ which gives a representative of the
orbit $O(z)$ of $z$ of the action.

The following lemma is given as Lemma 4.1 of \cite{kimura-2}.
\begin{lem}
\label{lem:Her-Ra-1}Let $\lm$ be a partition of $3$. For any $z\in Z_{\lm}$,
we can take a representative $\bx\in Z_{\lm}$ of the orbit $O(z)$
as given in the table: 

\bigskip

\begin{tabular}{|c|c|}
\hline 
$\lm$ & normal form $\bx$\tabularnewline
\hline 
\hline 
$(1,1,1)$ & $\left(\begin{array}{ccc}
1_{r} & 0 & 1_{r}\\
0 & 1_{r} & -1_{r}
\end{array}\right)$\tabularnewline
\hline 
$(2,1)$ & $\left(\begin{array}{ccc}
1_{r} & 0 & 0\\
0 & 1_{r} & 1_{r}
\end{array}\right)$\tabularnewline
\hline 
$(3)$ & $\left(\begin{array}{ccc}
1_{r} & 0 & 0\\
0 & 1_{r} & 0
\end{array}\right)$\tabularnewline
\hline 
\end{tabular}

\bigskip
\end{lem}

The representative $\bx\in Z_{\lm}$ given in Lemma \ref{lem:Her-Ra-1}
is called a \emph{normal form} of $z\in Z_{\lm}$. For the normal
form $\bx\in Z_{\lm}$, we write down the integral for the corresponding
Radon HGF in the case $\lm=(1,1,1),(2,1)$. The space of integration
variables is $T=\gras(r,2r)\simeq\GL r\backslash\mat'(r,2r)$ with
the homogeneous coordinates $t=(t',t'')\in\mat'(r,2r)$, $t',t''\in\mat(r)$.
Let $\{[t]\in T\mid\det t'\neq0\}$ be an affine neighbourhood of
$T$ and let $u=(u_{i,j})\in\mat(r)$ be the affine coordinates defined
by $t=t'(1_{r},u)$, namely $u=(t')^{-1}t''$. Then we know that 
\[
F_{\lm}(z,\al;C)=\int_{C}\chi_{\lm}(\vec{u}z;\al)du,\quad\vec{u}=(1_{r},u),\;du=\bigwedge_{1\leq i,j\leq r}du_{i,j}.
\]
Then we have 

\begin{align}
F_{(1,1,1)}(\bx,\al;C) & =\int_{C}(\det u)^{\al^{(2)}}(\det(1-u))^{\al^{(3)}}du,\label{eq:her-ra-3}\\
F_{(2,1)}(\bx,\al;C) & =\int_{C}e^{\al_{1}^{(1)}\Tr(u)}(\det u)^{\al_{0}^{(2)}}du,\label{eq:her-ra-4}
\end{align}
where the parameters $\al=(\al^{(1)},\al^{(2)},\al^{(3)})\in\C^{3}$
for $\lm=(1,1,1)$ satisfies
\begin{equation}
\al^{(1)}+\al^{(2)}+\al^{(3)}=-2r\label{eq:her-ra-5}
\end{equation}
and the parameters $\al=(\al_{0}^{(1)},\al_{1}^{(1)},\al_{0}^{(2)})\in\C^{3}$
for $\lm=(2,1)$ satisfies
\begin{equation}
\al_{0}^{(1)}+\al_{0}^{(2)}=-2r.\label{eq:her-ra-6}
\end{equation}
We assume $\al_{1}^{(1)}=-1$ in addition in case $\lm=(2,1)$. For
example we check the expression for $F_{(2,1)}(\bx,\al;C)$. Note
that the character is 
\begin{align*}
\chi_{(2,1)}(h;\al) & =(\det h_{0}^{(1)})^{\al_{0}^{(1)}}\exp\left(\al_{1}^{(1)}\Tr\left((h_{0}^{(1)})^{-1}h_{1}^{(1)}\right)\right)\cdot(\det h_{0}^{(2)})^{\al_{0}^{(2)}}
\end{align*}
and 
\[
\vec{u}\bx=(1_{r},u)\left(\begin{array}{ccc}
1_{r} & 0 & 0\\
0 & 1_{r} & 1_{r}
\end{array}\right)=(1_{r},u,u).
\]
Hence
\[
\chi_{(2,1)}(\vec{u}\bx;\al)=(\det1_{r})^{\al_{0}^{(1)}}e^{\al_{1}^{(1)}\Tr(u)}(\det u)^{\al_{0}^{(2)}}=e^{\al_{1}^{(1)}\Tr(u)}(\det u)^{\al_{0}^{(2)}}
\]
is the integrand of $F_{(2,1)}(\bx,\al;C)$. Note that we can attribute
$\al_{0}^{(2)}$ an arbitrary value, $a-r$ for example, since the
condition for $\al_{0}^{(2)}$ is only $\al_{0}^{(1)}+\al_{0}^{(2)}=-2r$.
By the condition $\al_{1}^{(1)}=-1$, the integrand has the same form
as the Hermitian matrix integral (\ref{eq:hermInt-4}) for $\G_{r}(a)$.
As for the space $\herm$ of integration variables for the Hermitian
matrix integral, we remark that $\herm$ is a real form of $\mat(r)$
in the sense that any $u\in\mat(r)$ can be written as $u=X+\sqrt{-1}Y$,
$X,Y\in\herm$ uniquely. Moreover we see from (\ref{eq:hermInt-2})
that the form $du$ coincides with $dU$ modulo constant when $u$
is restricted to $\herm$. Hence in the integrals (\ref{eq:her-ra-3})
and (\ref{eq:her-ra-4}), we can choose $r^{2}$-chain in the space
$\herm$. Hence the restriction $F_{\lm}(\bx,\al;C)$ of the Radon
HGF to the normal form $\bx$ gives the beta function (\ref{eq:herInt-3})
and the gamma function (\ref{eq:hermInt-4}) corresponding to the
partitions $(1,1,1)$ and $(2,1)$, respectively. The correspondence
of parameters is 

\begin{align*}
(\al^{(1)},\al^{(2)},\al^{(3)}) & =(-a-b,a-r,b-r)\quad\text{for}\quad\lm=(1,1,1),\\
(\al_{0}^{(1)},\al_{0}^{(1)},\al_{0}^{(2)}) & =(-a-r,-1,a-r)\quad\text{for}\quad\lm=(2,1).
\end{align*}
 For the beta and gamma functions $B_{r}(a,b),\G_{r}(a)$, the following
result is known \cite{Faraut}.
\begin{prop}
\label{prop:Her-be-ga}We have 
\begin{align*}
(1)\quad & \G_{r}(a)=\pi^{r(r-1)/2}\G(a)\G(a-1)\cdots\G(a-r+1),\\
(2)\quad & B_{r}(a,b)=\frac{\G_{r}(a)\G_{r}(b)}{\G_{r}(a+b)},
\end{align*}
where $\G(a)$ is the classical gamma function. 
\end{prop}

It is seen from this proposition that $\G_{r}(a)$ satisfies the contiguity
relation 
\[
\G_{r}(a+1)=a(a-1)\cdots(a-r+1)\G_{r}(a).
\]
Similar result can be obtained also for $B_{r}(a,b)$. In the following
we derive the above contiguity relations by applying Theorems \ref{thm:conti-nonconf}
and \ref{thm:cont-conf} to the Radon HGF corresponding to $\lm=(1,1,1)$
and $\lm=(2,1)$ without using Proposition \ref{prop:Her-be-ga}.

\subsection{Contiguity for the Radon beta}

We derive the contiguity relations for $B_{r}(a,b)$ using the contiguity
relation for the Radon HGF of type $\lm=(1,1,1)$ given in Theorem
\ref{thm:conti-nonconf}. In this subsection we write $F(z;\al)$
for $F_{(1,1,1)}(z;\al)$. Recall that the contiguity relation for
a root $\ep^{(i)}-\ep^{(j)}$ is given by 
\begin{equation}
\cL^{(i,j)}F(z,\al)=\bb(\al^{(j)})F(z,\al+\ep^{(i)}-\ep^{(j)}),\label{eq:Beta-1}
\end{equation}
where $\bb(s)=s(s+1)\cdots(s+r-1)$ is the $b$-function and the differential
operator $\cL^{(i,j)}$ of order $r$ is defined by
\begin{equation}
\cL^{(i,j)}=\det(\tr z^{(i)}\pa^{(j)})\label{eq:Beta-2}
\end{equation}
with $\pa^{(j)}=(\pa/\pa z_{a.b}^{(j)})_{0\leq a<2r,1\leq b\leq r}$.
\begin{prop}
\label{prop:Beta-0}For the case $\lm=(1,1,1)$ the contiguity relation
(\ref{eq:Beta-1}) for the Radon HGF gives 
\begin{align}
B_{r}(a+1,b) & =\frac{a(a-1)\cdots(a-r+1)}{(a+b)(a+b-1)\cdots(a+b-r+1)}B_{r}(a,b),\label{eq:Beta-3}\\
B_{r}(a,b+1) & =\frac{b(b-1)\cdots(b-r+1)}{(a+b)(a+b-1)\cdots(a+b-r+1)}B_{r}(a,b),\label{eq:Beta-4}
\end{align}
or the relations derived from the above two.
\end{prop}

For the proof of this result, we use the following lemma.
\begin{lem}
\label{lem:Beta-1}For any $z=(z^{(1)},z^{(2)},z^{(3)})\in Z_{(1,1,1)}$,
there exist $g\in\GL{2r}$ and $h\in H_{(1,1,1)}$ such that $z=g\bx h$
with 

\[
\bx=\left(\begin{array}{ccc}
1_{r} & 0 & 1_{r}\\
0 & 1_{r} & -1_{r}
\end{array}\right),
\]
where 
\[
g=(z^{(1)},z^{(2)})\diag(1_{r},-v_{1}^{-1}v_{2}),\quad h=\diag(1_{r},-v_{2}^{-1}v_{1},v_{1})
\]
with $v_{1},v_{2}\in\GL r$ being given by
\begin{equation}
v:=\left(\begin{array}{c}
v_{1}\\
v_{2}
\end{array}\right)=(z^{(1)},z^{(2)})^{-1}z^{(3)}.\label{eq:beta-1}
\end{equation}
\end{lem}

\begin{proof}
Noting that $z\in Z_{(1,1,1)}$ implies $\det(z^{(p)},z^{(q)})\neq0$
for any $1\leq p\neq q\leq3$, put $g_{1}=(z^{(1)},z^{(2)})\in\GL{2r}$
and consider 
\[
g_{1}^{-1}z=\left(\begin{array}{ccc}
1_{r} & 0 & v_{1}\\
0 & 1_{r} & v_{2}
\end{array}\right),\quad v=\left(\begin{array}{c}
v_{1}\\
v_{2}
\end{array}\right)=g_{1}^{-1}z^{(3)}.
\]
Note that $g_{1}^{-1}z\in Z_{(1,1,1)}$ and hence $v_{1},v_{2}\in\GL r$.
Take $h=\diag(1_{r},h_{2},h_{3})\in H_{(1,1,1)}$, then
\[
g_{1}^{-1}zh^{-1}=\left(\begin{array}{ccc}
1_{r} & 0 & v_{1}h_{3}^{-1}\\
0 & h_{2}^{-1} & v_{2}h_{3}^{-1}
\end{array}\right).
\]
We take $g_{2}=\diag(1_{r},h_{2}^{-1})\in\GL{2r}$ and consider 
\[
g_{2}^{-1}g_{1}^{-1}zh^{-1}=\left(\begin{array}{ccc}
1_{r} & 0 & v_{1}h_{3}^{-1}\\
0 & 1_{r} & h_{2}v_{2}h_{3}^{-1}
\end{array}\right).
\]
So we determine $h_{2},h_{3}$ by $v_{1}h_{3}^{-1}=1_{r},h_{2}v_{2}h_{3}^{-1}=-1_{r}$
so that $g_{2}^{-1}g_{1}^{-1}zh^{-1}$ becomes a normal form $\bx$,
namely 
\[
h_{2}=-v_{1}v_{2}^{-1},\quad h_{3}=v_{1}.
\]
Put $g=g_{1}g_{2}\in\GL{2r}$. Then we have $z=g\bx h$ with 
\[
g=(z^{(1)},z^{(2)})\diag(1_{r},-v_{2}v_{1}^{-1}),\quad h=\diag(1_{r},-v_{1}v_{2}^{-1},v_{1}).
\]
\end{proof}
By Lemma \ref{lem:Beta-1}, we have
\begin{equation}
F(z,\al)=F(g\bx h,\al)=(\det g)^{-r}\chi(h;\al)F(\bx,\al)=U(z;\al)F(\bx,\al),\label{eq:Beta-4-1}
\end{equation}
where 
\begin{equation}
U(z;\al)=\det(z^{(1)},z^{(2)})^{-r}(\det v_{1})^{\al^{(2)}+\al^{(3)}+r}(\det(-v_{2}))^{-\al^{(2)}-r}.\label{eq:Beta-5}
\end{equation}
Note that $F(\bx,\al)$ is independent of the variable $z$. Theorem
\ref{thm:conti-nonconf} says that we have the contiguity relation
(\ref{eq:Beta-1}) for the Radon HGF $F(z;\al)$ of type $\lm=(1,1,1)$
for any root $\ep^{(i)}-\ep^{(j)}\quad(1\leq i\neq j\leq3)$. Our
assertion (Proposition \ref{prop:Beta-0}) is that the contiguity
relations for $\ep^{(i)}-\ep^{(j)}$ gives either (\ref{eq:Beta-3})
or (\ref{eq:Beta-4}) or the combination of them. We check this assertion
for the roots $\ep^{(i)}-\ep^{(j)}\quad(1\leq i<j\leq3)$. Other cases
are checked in a similar way.

\subsubsection{Proof for the case $\protect\ep^{(1)}-\protect\ep^{(3)}$}

The contiguity relation for the Radon HGF in this case is 
\begin{equation}
\cL^{(1,3)}F(z,\al)=\bb(\al^{(3)})F(z,\al+\ep^{(1)}-\ep^{(3)})\label{eq:Beta-5-1}
\end{equation}
with $\cL^{(1,3)}=\det(\tr z^{(1)}\pa^{(3)})$. We compute the left
hand side of (\ref{eq:Beta-5-1}). Note that
\[
\cL^{(1,3)}F(z,\al)=\cL^{(1,3)}U(z;\al)\cdot F(\bx,\al)
\]
since $\bx$ is a constant matrix, where $U(z;\al)$ is that given
by (\ref{eq:Beta-5}). So we compute $\cL^{(1,3)}U(z;\al)$. Looking
at the form (\ref{eq:Beta-5}) of $U(z;\al)$, we introduce the new
variables 
\[
(\tilde{z}^{(1)},\tilde{z}^{(2)},v):=(z^{(1)},z^{(2)},(z^{(1)},z^{(2)})^{-1}z^{(3)})
\]
and express the operator $\cL^{(1,3)}$ in terms of them. Put
\[
v=(v_{a,b})_{0\leq a<2r,1\leq b\leq r},\quad v_{1}=(v_{a,b})_{0\leq a<r,1\leq b\leq r},\quad v_{2}=(v_{a,b})_{r\leq a<2r,1\leq b\leq r}.
\]
Accordingly, we write 
\[
\pa_{v}:=(\pa/\pa v_{a,b})_{0\leq a<2r,1\leq b\leq r},\quad\pa_{v_{1}}:=(\pa/\pa v_{a,b})_{0\leq a<r,1\leq b\leq r},\quad\pa_{v_{2}}:=(\pa/\pa v_{a,b})_{r\leq a<2r,1\leq b\leq r}.
\]

\begin{lem}
The operator $\cL^{(1,3)}$ is written as $\det(\pa_{v_{1}})$ in
the variables $(\tilde{z}^{(1)},\tilde{z}^{(2)},v)$.
\end{lem}

\begin{proof}
Note that only $v$ depends on $z^{(3)}$ among the new variables.
Since the change of the operator $\pa^{(3)}\mapsto\pa_{v}$ is contravariant
to that of $z^{(3)}\mapsto v$, we have $\pa^{(3)}=\tr(z^{(1)},z^{(2)})^{-1}\pa_{v}.$
It follows that 
\begin{align*}
\tr z^{(1)}\pa^{(3)} & =\tr z^{(1)}\tr(z^{(1)},z^{(2)})^{-1}\pa_{v}=\tr\left((z^{(1)},z^{(2)})^{-1}z^{(1)}\right)\pa_{v}\\
 & =(1_{r},0_{r})\pa_{v}=\pa_{v_{1}}
\end{align*}
and $\cL^{(1,3)}=\det(\tr z^{(1)}\pa^{(3)})=\det(\pa_{v_{1}})$. 
\end{proof}
Using the above lemma, we have 
\begin{align*}
\cL^{(1,3)}U(z;\al) & =\det(\pa_{v_{1}})U(z;\al)\\
 & =\det(z^{(1)},z^{(2)})^{-r}\cdot\det(\pa_{v_{1}})(\det v_{1})^{\al^{(2)}+\al^{(3)}+r}\cdot(\det(-v_{2}))^{-\al^{(2)}-r}\\
 & =\det(z^{(1)},z^{(2)})^{-r}\cdot\bb(\al^{(2)}+\al^{(3)}+r)(\det v_{1})^{\al^{(2)}+\al^{(3)}+r-1}\cdot(\det(-v_{2}))^{-\al^{(2)}-r}\\
 & =\bb(\al^{(2)}+\al^{(3)}+r)U(z;\al+\ep^{(1)}-\ep^{(3)})
\end{align*}
by virtue of Cayley's formula. Since the right hand side of (\ref{eq:Beta-5-1})
is $\bb(\al^{(3)})U(z;\al+\ep^{(1)}-\ep^{(3)})F(\bx,\al+\ep^{(1)}-\ep^{(3)})$,
(\ref{eq:Beta-5-1}) reads as
\[
\bb(\al^{(2)}+\al^{(3)}+r)F(\bx,\al)=\bb(\al^{(3)})F(\bx,\al+\ep^{(1)}-\ep^{(3)}).
\]
Putting $\al^{(2)}=a-r,\al^{(3)}=b-r$, this can be written as 
\[
B_{r}(a,b)=\frac{(b-1)\cdots(b-r)}{(a+b-1)\cdots(a+b-r)}B_{r}(a,b-1).
\]
This gives the formula (\ref{eq:Beta-4}) for the beta function.

\subsubsection{Proof for the case $\protect\ep^{(1)}-\protect\ep^{(2)}$}

The contiguity relation for the Radon HGF in this case is 
\begin{equation}
\cL^{(1,2)}F(z,\al)=\bb(\al^{(2)})F(z,\al+\ep^{(1)}-\ep^{(2)}).\label{eq:Beta-5-2}
\end{equation}
with $\cL^{(1,2)}=\det(\tr z^{(1)}\pa^{(2)})$. The left hand side
of (\ref{eq:Beta-5-2}) is $\cL^{(1,2)}F(z,\al)=D^{(1,2)}U(z;\al)\cdot F(\bx,\al)$,
where $U(z;\al)$ is that given by (\ref{eq:Beta-5}). To compute
$\cL^{(1,2)}U(z;\al)$ let us write the operator $\cL^{(1,2)}$ in
terms of the new coordinates 
\[
(\tilde{z}^{(1)},\tilde{z}^{(2)},v):=(z^{(1)},z^{(2)},(z^{(1)},z^{(2)})^{-1}z^{(3)})
\]
as in the previous case. 
\begin{lem}
We have $\cL^{(1,2)}U(z;\al)=(-1)^{r}\det v_{2}\cdot\det(\pa_{v_{1}})U(z;\al)$.
\end{lem}

\begin{proof}
Since the new variables relating to $z^{(2)}$ are $\tilde{z}^{(2)}$
and $v$, the matrix of derivations $\pa^{(2)}=(\pa_{a,b}^{(2)})_{0\leq a<2r,1\leq b\leq r}$,
$\pa_{a,b}^{(2)}:=\pa/\pa z_{a,b}^{(2)}$ can be written symbolically
as $\pa^{(2)}=\tilde{\pa}^{(2)}+(\pa v/\pa z^{(2)})\pa_{v}$, where
$\tilde{\pa}^{(2)}=(\tilde{\pa}_{a,b}^{(2)})_{0\leq a<2r,1\leq b\leq r},\tilde{\pa}_{a,b}^{(2)}:=\pa/\pa\tilde{z}_{a,b}^{(2)}$.
We compute the second term. Since
\begin{align}
\frac{\pa v}{\pa z_{a,b}^{(2)}} & =\frac{\pa}{\pa z_{a,b}^{(2)}}(z^{(1)},z^{(2)})^{-1}z^{(3)}\nonumber \\
 & =-(z^{(1)},z^{(2)})^{-1}\cdot\frac{\pa}{\pa z_{a,b}^{(2)}}(z^{(1)},z^{(2)})\cdot(z^{(1)},z^{(2)})^{-1}z^{(3)}\nonumber \\
 & =-(z^{(1)},z^{(2)})^{-1}\cdot(0_{2r,r},E_{a,b})\cdot v,\label{eq:beta-3}
\end{align}
where $0_{2r,r}$ is the zero matrix in $\mat(2r,r)$ and $E_{a,b}\in\mat(2r,r)$
is the $(a,b)$-th matrix unit, namely, the matrix unit whose only
nonzero number, $1$, locates at the $(a,b)$-th entry, we have 
\[
\pa_{a,b}^{(2)}=\tilde{\pa}_{a,b}^{(2)}+\sum_{0\leq k<2r,1\leq l\leq r}\left(\frac{\pa v}{\pa z_{a,b}^{(2)}}\right)_{k,l}(\pa_{v})_{k,l}=\tilde{\pa}_{a,b}^{(2)}+\Tr\left\{ \tr\left(\frac{\pa v}{\pa z_{a,b}^{(2)}}\right)\pa_{v}\right\} .
\]
Then 
\begin{align}
(\tr z^{(1)}\pa^{(2)})_{p,q} & =\sum_{k=1}^{2r}(\tr z^{(1)})_{p,k}\pa_{k,q}^{(2)}\nonumber \\
 & =(\tr z^{(1)}\tilde{\pa}^{(2)})_{p,q}+\sum_{0\leq k<2r}(\tr z^{(1)})_{p,k}\Tr\left\{ \tr\left(\frac{\pa v}{\pa z_{k,q}^{(2)}}\right)\pa_{v}\right\} \label{eq:Beta-6}
\end{align}
Using (\ref{eq:beta-3}), the second term of (\ref{eq:Beta-6}) is
written as 

\begin{align*}
 & \sum_{0\leq k<2r}(\tr z^{(1)})_{p,k}\Tr\left\{ \tr\left(\frac{\pa v}{\pa z_{k,q}^{(2)}}\right)\pa_{v}\right\} \\
 & =\sum_{0\leq k<2r}(\tr z^{(1)})_{p,k}\Tr\left\{ \tr\left(-(z^{(1)},z^{(2)})^{-1}\cdot(0_{2r,r},E_{k,q})\cdot v\right)\pa_{v}\right\} \\
 & =\Tr\left\{ \tr\left(-(z^{(1)},z^{(2)})^{-1}\cdot\sum_{0\leq k<2r}z_{k,p}^{(1)}(0_{2r,r},E_{k,q})\cdot v\right)\pa_{v}\right\} \\
 & =\Tr\left\{ \tr\left(-(z^{(1)},z^{(2)})^{-1}\cdot(0_{2r,r},0,\dots,0,z_{p}^{(1)},0,\dots,0)\cdot v\right)\pa_{v}\right\} \\
 & =-\Tr\left\{ \tr\left(\left(\begin{array}{cc}
0_{r} & E_{p,q}\\
0_{r} & 0_{r}
\end{array}\right)\cdot v\right)\pa_{v}\right\} =-\Tr\left\{ (\tr v_{2}E_{q,p},0_{r})\left(\begin{array}{c}
\pa_{v_{1}}\\
\pa_{v_{2}}
\end{array}\right)\right\} \\
 & =-\Tr\left\{ \tr v_{2}E_{q,p}\pa_{v_{1}}\right\} =-\left(\pa_{v_{1}}(\tr v_{2})\right)_{p,q}.
\end{align*}
On the other hand, for the first part of (\ref{eq:Beta-6}) we have
$(\tr z^{(1)}\tilde{\pa}^{(2)})_{p,q}U(z;\al)=0$ for any $p,q$.
In fact, noting $(\tilde{z}^{(1)},\tilde{z}^{(2)})=(z^{(1)},z^{(2)})$,
\begin{align*}
(\tr z^{(1)}\tilde{\pa}^{(2)})_{p,q}\det(z^{(1)},z^{(2)}) & =\sum_{0\leq k<2r}(\tr z^{(1)})_{p,k}\pa_{k,q}^{(2)}\det(z^{(1)},z^{(2)})\\
 & =\det(z^{(1)},z_{1}^{(2)},\dots,z_{p}^{(1)},\dots,z_{r}^{(2)})=0.
\end{align*}
It follows that $(\tr z^{(1)}\tilde{\pa}^{(2)})_{p,q}\det(z^{(1)},z^{(2)})^{-r}=0$
for any $p,q$. Hence when we apply $\cL^{(1,2)}=\det(\tr z^{(1)}\pa^{(2)})$
to $U(z;\al)$, it acts as the operator $(-1)^{r}\det v_{2}\cdot\det(\pa_{v_{1}})$. 
\end{proof}
By the above lemma, we have 
\begin{align*}
\cL^{(1,2)}U(z;\al) & =(-1)^{r}\det v_{2}\cdot\det(\pa_{v_{1}})\det(z^{(1)},z^{(2)})^{-r}(\det v_{1})^{\al^{(2)}+\al^{(3)}+r}(\det(-v_{2}))^{-\al^{(2)}-r}\\
 & =\bb(\al^{(2)}+\al^{(3)}+r)\det(z^{(1)},z^{(2)})^{-r}(\det v_{1})^{\al^{(2)}+\al^{(3)}+r-1}(\det(-v_{2}))^{-(\al^{(2)}-1)-r}\\
 & =\bb(\al^{(2)}+\al^{(3)}+r)U(z;\al+\ep^{(1)}-\ep^{(2)})
\end{align*}
by virtue of Cayley's formula. Noting that the right hand side of
(\ref{eq:Beta-5-2}) is $\bb(\al^{(2)})U(z;\al+\ep^{(1)}-\ep^{(2)})F(\bx,\al+\ep^{(1)}-\ep^{(2)})$,
from (\ref{eq:Beta-5-2}), we have
\begin{equation}
\bb(\al^{(2)}+\al^{(3)}+r)F(\bx,\al)=\bb(\al^{(2)})F(\bx,\al+\ep^{(1)}-\ep^{(2)}).\label{eq:Beta-6-1}
\end{equation}
Putting $\al^{(2)}=a-r,\al^{(3)}=b-r$, (\ref{eq:Beta-6-1}) can be
written as 
\[
B_{r}(a,b)=\frac{(a-1)\cdots(a-r)}{(a+b-1)\cdots(a+b-r)}B_{r}(a-1,b).
\]
This is the formula (\ref{eq:Beta-3}) in Proposition \ref{prop:Beta-0}.

\subsubsection{Proof for the case $\protect\ep^{(2)}-\protect\ep^{(3)}$}

The contiguity relation for the Radon HGF in this case is 
\begin{equation}
\cL^{(2,3)}F(z,\al)=\bb(\al^{(3)})F(z,\al+\ep^{(2)}-\ep^{(3)}).\label{eq:Beta-7}
\end{equation}
with $\cL^{(2,3)}=\det(\tr z^{(2)}\pa^{(3)})$. The left hand side
of (\ref{eq:Beta-7}) is $\cL^{(2,3)}F(z,\al)=\cL^{(2,3)}U(z;\al)\cdot F(\bx,\al)$,
where $U(z;\al)$ is that given by (\ref{eq:Beta-5}). To compute
$\cL^{(2,3)}U(z;\al)$, let us write the operator $\cL^{(2,3)}$ in
terms of the new variables
\[
(\tilde{z}^{(1)},\tilde{z}^{(2)},v):=(z^{(1)},z^{(2)},(z^{(1)},z^{(2)})^{-1}z^{(3)})
\]
as in the previous two cases. 
\begin{lem}
The operator $\cL^{(2,3)}$ is written as $\det(\pa_{v_{2}})$ in
the variables $(\tilde{z}^{(1)},\tilde{z}^{(2)},v)$.
\end{lem}

\begin{proof}
Since the change of the operator $\pa^{(3)}\mapsto\pa_{v}$ is contravariant
to that of $z^{(3)}\mapsto v$, we have $\pa^{(3)}=\tr(z^{(1)},z^{(2)})^{-1}\pa_{v}$.
It follows that 
\begin{align*}
\tr z^{(2)}\pa^{(3)} & =\tr z^{(2)}\tr(z^{(1)},z^{(2)})^{-1}\pa_{v}=\tr\left((z^{(1)},z^{(2)})^{-1}z^{(2)}\right)\pa_{v}\\
 & =(0_{r},1_{r})\pa_{v}=\pa_{v_{2}}.
\end{align*}
Hence we have $\cL^{(2,3)}=\det(\tr z^{(2)}\pa^{(3)})=\det(\pa_{v_{2}})$.
\end{proof}
We apply the above lemma to $U(z;\al)$ and assert that 
\begin{equation}
\cL^{(2,3)}U(z;\al)=(-1)^{r}\bb(-\al^{(2)}-r)U(z;\al+\ep^{(2)}-\ep^{(3)}).\label{eq:Beta-8}
\end{equation}
To see this we first note that 
\[
\cL^{(2,3)}U(z;\al)=\det(z^{(1)},z^{(2)})^{-r}(\det v_{1})^{\al^{(2)}+\al^{(3)}+r}\cdot\det(\pa_{v_{2}})(\det(-v_{2}))^{-\al^{(2)}-r}
\]
and show that
\begin{equation}
\det(\pa_{v_{2}})\det(-v_{2})^{-\al^{(2)}-r}=(-1)^{r}\bb(-\al^{(2)}-r)\det(-v_{2})^{-\al^{(2)}-r-1}.\label{eq:beta-2}
\end{equation}
In fact, if we make a change of variables $v_{2}\mapsto w:=-v_{2}$,
we have $\det(\pa_{v_{2}})=(-1)^{r}\det(\pa_{w})$ and 
\begin{align*}
\det(\pa_{v_{2}})\det(-v_{2})^{-\al^{(2)}-r} & =(-1)^{r}\det(\pa_{w})\det(w)^{-\al^{(2)}-r}\\
 & =(-1)^{r}\bb(-\al^{(2)}-r)\det(w)^{-\al^{(2)}-r-1}.
\end{align*}
by virtue of Cayley's formula. Thus we have (\ref{eq:Beta-8}). 

Now using (\ref{eq:Beta-8}), the contiguity relation (\ref{eq:Beta-7})
gives 
\[
(-1)^{r}\bb(-\al^{(2)}-r)F(\bx,\al)=\bb(\al^{(3)})F(\bx,\al+\ep^{(2)}-\ep^{(3)}).
\]
Putting $\al^{(2)}=a-r,\al^{(3)}=b-r$ in the both sides, we have
\[
a(a-1)\cdots(a-r+1)B_{r}(a,b)=(b-1)(b-2)\cdots(b-r)B_{r}(a+1,b-1),
\]
which is derived from (\ref{eq:Beta-3}) and (\ref{eq:Beta-4}) as
\begin{align*}
 & (b-1)(b-2)\cdots(b-r)B_{r}(a+1,b-1)\\
 & =(b-1)(b-2)\cdots(b-r)\frac{a(a-1)\cdots(a-r)}{(a+b-1)\cdots(a+b-r)}B_{r}(a,b-1)\\
 & =(b-1)(b-2)\cdots(b-r)\frac{a(a-1)\cdots(a-r)}{(a+b-1)\cdots(a+b-r)}\frac{(a+b-1)\cdots(a+b-r)}{(b-1)(b-2)\cdots(b-r)}B_{r}(a,b)\\
 & =a(a-1)\cdots(a-r+1)B_{r}(a,b).
\end{align*}

\subsection{Contiguity for Radon gamma}

We derive the contiguity relations for the gamma function $\G_{r}(a)$
given by the Hermitian matrix integral (\ref{eq:hermInt-4}) using
the contiguity relations for the Radon HGF of type $\lm=(2,1)$ given
in Theorem \ref{thm:cont-conf} . In this subsection we write $F(z;\al)$
for $F_{(2,1)}(z;\al)$ for the sake of brevity. Recall that the contiguity
relation for $F(z;\al)$ in this case is given by 

\begin{align}
\cL^{(1,2)}F(z,\al) & =\bb(\al_{0}^{(2)})F(z,\al+\ep^{(1)}-\ep^{(2)}),\label{eq:gamma-0}\\
\cL^{(2,1)}F(z,\al) & =(-1)^{r}F(z,\al+\ep^{(2)}-\ep^{(1)}),\label{eq:gamma-1}
\end{align}
where $\bb(s)=s(s+1)\cdots(s+r-1)$ is the $b$-function, the differential
operators $\cL^{(1,2)},\cL^{(2,1)}$ of order $r$ are given by 
\[
\cL^{(1,2)}=\det(\tr z_{0}^{(1)}\pa_{0}^{(2)}),\quad\cL^{(2,1)}=\det(\tr z_{0}^{(2)}\pa_{1}^{(1)})
\]
with the matrices
\[
\pa_{0}^{(2)}=(\pa/\pa(z_{0}^{(2)})_{a,b})_{0\leq a<2r,1\leq b\leq r},\quad\pa_{1}^{(1)}=(\pa/\pa(z_{1}^{(1)})_{a,b})_{0\leq a<2r,1\leq b\leq r},
\]
and $\al\mapsto\al+\ep^{(1)}-\ep^{(2)}$ means the change $\al_{0}^{(1)}\mapsto\al_{0}^{(1)}+1,\al_{0}^{(2)}\mapsto\al_{0}^{(2)}-1$
in $\al$.
\begin{prop}
For the case $\lm=(2,1)$, the contiguity relations (\ref{eq:gamma-0})
and (\ref{eq:gamma-1}) for the Radon HGF of type $\lm$ give the
contiguity for the gamma function:
\[
\G_{r}(a+1)=a(a-1)\cdots(a-r+1)\G_{r}(a).
\]
\end{prop}

For the proof of this result, we use the following lemma.
\begin{lem}
\label{lem:Gamma-1}For any $z=(z_{0}^{(1)},z_{1}^{(1)},z_{0}^{(2)})\in Z_{(2,1)}$,
there exist $g\in\GL{2r}$ and $h\in H_{(2,1)}$ such that

\[
z=g\bx h,\quad\bx=\left(\begin{array}{ccc}
1_{r} & 0 & 0\\
0 & 1_{r} & 1_{r}
\end{array}\right)\in Z_{(2,1)}.
\]
There are two choice for the above $(g,h)$:

(1) $g=(z_{0}^{(1)},z_{1}^{(1)})\left(\begin{array}{cc}
1_{r} & v_{1}v_{2}^{-1}\\
 & 1_{r}
\end{array}\right),h=\left(\begin{array}{cc}
1_{r} & -v_{1}v_{2}^{-1}\\
 & 1_{r}
\end{array}\right)\oplus v_{2}$ with 
\begin{equation}
v=\left(\begin{array}{c}
v_{1}\\
v_{2}
\end{array}\right)=(z_{0}^{(1)},z_{1}^{(1)})^{-1}z_{0}^{(2)}.\label{eq:gamma-2}
\end{equation}

(2) $g=(z_{0}^{(1)},z_{0}^{(2)})\diag(v_{2}^{-1},1_{r}),h=\left(\begin{array}{cc}
v_{2} & v_{2}v_{1}\\
 & v_{2}
\end{array}\right)\oplus1_{r}$ with 
\begin{equation}
v=\left(\begin{array}{c}
v_{1}\\
v_{2}
\end{array}\right)=(z_{0}^{(1)},z_{0}^{(2)})^{-1}z_{1}^{(1)}.\label{eq:gamma-3}
\end{equation}
\end{lem}

\begin{proof}
Note that $z\in Z_{(2,1)}$ implies  $\det(z_{0}^{(1)},z_{1}^{(1)})\neq0$
and $\det(z_{0}^{(1)},z_{0}^{(2)})\ne0$. We show that the reduction
$z\to\bx$ is carried out by the choice of $(g,h)$ given by (1).
Put $g_{1}=(z_{0}^{(1)},z_{1}^{(1)})\in\GL{2r}$ and consider 
\[
g_{1}^{-1}z=\left(\begin{array}{ccc}
1_{r} & 0 & v_{1}\\
0 & 1_{r} & v_{2}
\end{array}\right),\quad v=\left(\begin{array}{c}
v_{1}\\
v_{2}
\end{array}\right)=g_{1}^{-1}z_{0}^{(2)}=(z_{0}^{(1)},z_{1}^{(1)})^{-1}z_{0}^{(2)}.
\]
Note that $g_{1}^{-1}z\in Z_{(2,1)}$ and hence $v_{2}\in\GL r$.
Take $h=\left(\begin{array}{cc}
1_{r} & h_{1}\\
 & 1_{r}
\end{array}\right)\oplus h_{2}\in H_{(2,1)}$ with $h_{2}\in\GL r$, then
\[
g_{1}^{-1}zh^{-1}=\left(\begin{array}{ccc}
1_{r} & -h_{1} & v_{1}h_{2}^{-1}\\
0 & 1_{r} & v_{2}h_{2}^{-1}
\end{array}\right).
\]
Take $h_{2}$ so that $v_{2}h_{2}^{-1}=1_{r}$, namely $h_{2}=v_{2}$.
Then 
\[
g_{1}^{-1}zh^{-1}=\left(\begin{array}{ccc}
1_{r} & -h_{1} & v_{1}v_{2}^{-1}\\
0 & 1_{r} & 1_{r}
\end{array}\right).
\]
Take $g_{2}^{-1}=\left(\begin{array}{cc}
1_{r} & -v_{1}v_{2}^{-1}\\
 & 1_{r}
\end{array}\right)$ and we have 
\[
g_{2}^{-1}g_{1}^{-1}zh^{-1}=\left(\begin{array}{ccc}
1_{r} & -h_{1}-v_{1}v_{2}^{-1} & 0\\
0 & 1_{r} & 1_{r}
\end{array}\right).
\]
Then determining $h_{1}$ as $h_{1}=-v_{1}v_{2}^{-1}$, we obtain
the desired normal form $\bx$. Put $g=g_{1}g_{2}\in\GL{2r}$. Then
we have $z=g\bx h$ with 
\[
g=(z_{0}^{(1)},z_{1}^{(1)})\left(\begin{array}{cc}
1_{r} & v_{1}v_{2}^{-1}\\
 & 1_{r}
\end{array}\right),\quad h=\left(\begin{array}{cc}
1_{r} & -v_{1}v_{2}^{-1}\\
 & 1_{r}
\end{array}\right)\oplus v_{2}.
\]
Next we consider the reduction $z\to\bx$ by (2). Put $g_{1}=(z_{0}^{(1)},z_{0}^{(2)})\in\GL{2r}$
and consider
\[
g_{1}^{-1}z=\left(\begin{array}{ccc}
1_{r} & v_{1} & 0\\
0 & v_{2} & 1_{r}
\end{array}\right),\quad v=\left(\begin{array}{c}
v_{1}\\
v_{2}
\end{array}\right)=g_{1}^{-1}z_{1}^{(1)}=(z_{0}^{(1)},z_{0}^{(2)})^{-1}z_{1}^{(1)}.
\]
Note that $g_{1}^{-1}z\in Z_{(2,1)}$ implies $v_{2}\in\GL r$. Put
$h=\left(\begin{array}{cc}
h_{0} & h_{1}\\
 & h_{0}
\end{array}\right)\oplus1_{r}\in H_{(2,1)}$, then
\[
g_{1}^{-1}zh^{-1}=\left(\begin{array}{ccc}
1_{r} & v_{1} & 0\\
0 & v_{2} & 1_{r}
\end{array}\right)\left(\begin{array}{ccc}
h_{0} & h_{1} & 0\\
 & h_{0}\\
 &  & 1_{r}
\end{array}\right)^{-1}=\left(\begin{array}{ccc}
h_{0}^{-1} & -h_{0}^{-1}h_{1}h_{0}^{-1}+v_{1}h_{0}^{-1} & 0\\
0 & v_{2}h_{0}^{-1} & 1_{r}
\end{array}\right).
\]
So we take $g_{2}^{-1}=\diag(h_{0},1_{r})$ and have 
\[
g_{2}^{-1}g_{1}^{-1}zh^{-1}=\left(\begin{array}{ccc}
1_{r} & -h_{1}h_{0}^{-1}+h_{0}v_{1}h_{0}^{-1} & 0\\
0 & v_{2}h_{0}^{-1} & 1_{r}
\end{array}\right).
\]
We determine $h_{0},h_{1}$ by the condition $v_{2}h_{0}^{-1}=1_{r},-h_{1}+h_{0}v_{1}=0$,
namely, $h_{0}=v_{2},h_{1}=v_{2}v_{1}$. Then $g_{2}^{-1}g_{1}^{-1}zh^{-1}$
becomes the desired normal form $\bx$. Put $g=g_{1}g_{2}\in\GL{2r}$.
Then we have $z=g\bx h$ with 
\[
g=(z_{0}^{(1)},z_{0}^{(2)})\left(\begin{array}{cc}
v_{2}^{-1}\\
 & 1_{r}
\end{array}\right),\quad h=\left(\begin{array}{cc}
v_{2} & v_{2}v_{1}\\
 & v_{2}
\end{array}\right)\oplus1_{r}.
\]
\end{proof}

\subsubsection{Proof for the case $\protect\ep^{(2)}-\protect\ep^{(1)}$}

The contiguity relation for the Radon HGF we use in this case is 
\begin{equation}
\cL^{(2,1)}F(z,\al)=(-1)^{r}F(z,\al+\ep^{(2)}-\ep^{(1)}).\label{eq:gamma-3-1}
\end{equation}
We use the normalization (2) of Lemma \ref{lem:Gamma-1}. Then 
\[
F(z,\al)=F(g\bx h,\al)=(\det g)^{-r}\chi(h;\al)F(\bx,\al)=U(z;\al)F(\bx,\al),
\]
where 
\[
U(z;\al)=\det(z_{0}^{(1)},z_{0}^{(2)})^{-r}(\det v_{2})^{\al_{0}^{(1)}+r}\exp(\al_{1}^{(1)}\Tr(v_{1}))
\]
and $v=(z_{0}^{(1)},z_{0}^{(2)})^{-1}z_{1}^{(1)}$. Noting that $\cL^{(2,1)}=\det\left(\tr z_{0}^{(2)}\pa_{1}^{(1)}\right)$,
we rewrite this operator in terms of new variables $(\tilde{z}_{0}^{(1)},v,\tilde{z}_{0}^{(2)})$
defined by 
\[
(\tilde{z}_{0}^{(1)},v,\tilde{z}_{0}^{(2)})=(z_{0}^{(1)},(z_{0}^{(1)},z_{0}^{(2)})^{-1}z_{1}^{(1)},z_{0}^{(2)}).
\]
 
\begin{lem}
The operator $\cL^{(2,1)}$ is written as $\det(\pa_{v_{2}})$ in
the variables $(\tilde{z}_{0}^{(1)},v,\tilde{z}_{0}^{(2)})$.
\end{lem}

\begin{proof}
Since $z_{1}^{(1)}=(z_{0}^{(1)},z_{0}^{(2)})v$, we have $\pa_{1}^{(1)}=\tr(z_{0}^{(1)},z_{0}^{(2)})^{-1}\pa_{v}$.
It follows that 
\begin{align*}
\tr z_{0}^{(2)}\pa_{1}^{(1)} & =\tr z_{0}^{(2)}\tr(z_{0}^{(1)},z_{0}^{(2)})^{-1}\pa_{v}=\tr((z_{0}^{(1)},z_{0}^{(2)})^{-1}z_{0}^{(2)})\pa_{v}\\
 & =\tr\left(\begin{array}{c}
0\\
1_{r}
\end{array}\right)\pa_{v}=\pa_{v_{2}}
\end{align*}
and $\cL^{(2,1)}=\det(\tr z_{0}^{(2)}\pa_{1}^{(1)})=\det(\pa_{v_{2}})$. 
\end{proof}
Using Cayley's formula, we have 
\begin{align*}
\cL^{(2,1)}U(z;\al) & =\det(\pa_{v_{2}})\det(z_{0}^{(1)},z_{0}^{(2)})^{-r}(\det v_{2})^{\al_{0}^{(1)}+r}\exp(\al_{1}^{(1)}\Tr(v_{1}))\\
 & =\bb(\al_{0}^{(1)}+r)\det(z_{0}^{(1)},z_{0}^{(2)})^{-r}(\det v_{2})^{\al_{0}^{(1)}+r-1}\exp(-\Tr(v_{1}))\\
 & =\bb(\al_{0}^{(1)}+r)U(z;\al+\ep^{(2)}-\ep^{(1)}).
\end{align*}
Then the contiguity relation (\ref{eq:gamma-3-1}) reads 
\[
\bb(\al_{0}^{(1)}+r)F(\bx;\al)=(-1)^{r}F(\bx;\al+\ep^{(2)}-\ep^{(1)}).
\]
Noting that $\al_{0}^{(1)}+\al_{0}^{(2)}=-2r$ and 
\begin{align*}
(-1)^{r}\bb(\al_{0}^{(1)}+r) & =(-1)^{r}\bb(-\al_{0}^{(2)}-r)=(-1)^{r}\bb(-a)\\
 & =(-1)^{r}(-a)(-a+1)\cdots(-a+r-1)\\
 & =a(a-1)\cdots(a-r+1),
\end{align*}
we obtain 
\[
\G_{r}(a+1)=a(a-1)\cdots(a-r+1)\G_{r}(a),
\]
which is the formula to be shown.

\subsubsection{Proof for the case $\protect\ep^{(1)}-\protect\ep^{(2)}$}

The contiguity relation for the Radon HGF in this case is 
\begin{equation}
\cL^{(1,2)}F(z,\al)=\bb(\al_{0}^{(2)})F(z,\al+\ep^{(1)}-\ep^{(2)}).\label{eq:gamma-3-2}
\end{equation}
We use the normalization (1) of Lemma \ref{lem:Gamma-1} in this case.
Then 
\[
F(z,\al)=F(g\bx h,\al)=(\det g)^{-r}\chi(h;\al)F(\bx,\al)=U(z;\al)F(\bx,\al),
\]
where
\begin{align*}
U(z;\al) & =\det(z_{0}^{(1)},z_{1}^{(1)})^{-r}\exp\left(\al_{1}^{(1)}\Tr(-v_{1}v_{2}^{-1})\right)(\det v_{2})^{\al_{0}^{(2)}}\\
 & =\det(z_{0}^{(1)},z_{1}^{(1)})^{-r}\exp\left(\Tr(v_{1}v_{2}^{-1})\right)(\det v_{2})^{\al_{0}^{(2)}}
\end{align*}
with $v=(z_{0}^{(1)},z_{1}^{(1)})^{-1}z_{0}^{(2)}$. Here we used
$\al_{1}^{(1)}=-1$. Noting that $\cL^{(1,2)}=\det\left(\tr z_{0}^{(1)}\pa_{0}^{(2)}\right)$,
we rewrite this operator in terms of new variables $(\tilde{z}_{0}^{(1)},\tilde{z}_{1}^{(1)},v)$,
where
\[
(\tilde{z}_{0}^{(1)},\tilde{z}_{1}^{(1)},v)=(z_{0}^{(1)},z_{1}^{(1)},(z_{0}^{(1)},z_{1}^{(1)})^{-1}z_{0}^{(2)}).
\]

\begin{lem}
The operator $\cL^{(1,2)}$ is written as $\det(\pa_{v_{1}})$ in
the variables $(\tilde{z}_{0}^{(1)},\tilde{z}_{1}^{(1)},v)$.
\end{lem}

\begin{proof}
Since $z_{0}^{(2)}=(z_{0}^{(1)},z_{1}^{(1)})v$, we have $\pa_{0}^{(2)}=\tr(z_{0}^{(1)},z_{1}^{(1)})^{-1}\pa_{v}$.
It follows that 
\begin{align*}
\tr z_{0}^{(1)}\pa_{0}^{(2)} & =\tr z_{0}^{(1)}\tr(z_{0}^{(1)},z_{1}^{(1)})^{-1}\pa_{v}=\tr((z_{0}^{(1)},z_{1}^{(1)})^{-1}z_{0}^{(1)})\pa_{v}\\
 & =\tr\left(\begin{array}{c}
1_{r}\\
0
\end{array}\right)\pa_{v}=\pa_{v_{1}}.
\end{align*}
\end{proof}
We need the identity 
\begin{equation}
\det(\pa_{v_{1}})\exp\left(\Tr(v_{1}A)\right)=(\det A)\exp\left(\Tr(v_{1}A)\right),\quad A\in\mat(r).\label{eq:gamma-4}
\end{equation}
To see this, put $v_{1}=(v_{i,j})$ and $\pa_{i,j}=\pa/\pa v_{i,j}$
and show 
\[
\pa_{i,j}\exp\left(\Tr(v_{1}A)\right)=\exp\left(\Tr(v_{1}A)\right)A_{j,i}.
\]
In fact, this formula can be checked as 
\begin{align*}
\pa_{i,j}\exp\left(\Tr(v_{1}A)\right) & =\exp\left(\Tr(v_{1}A)\right)\pa_{i,j}\Tr(v_{1}A)=\exp\left(\Tr(v_{1}A)\right)\Tr(E_{i,j}A)\\
 & =\exp\left(\Tr(v_{1}A)\right)\sum_{a,b}(E_{i,j})_{a,b}A_{b,a}=\exp\left(\Tr(v_{1}A)\right)\sum_{a,b}\de_{i,a}\de_{j,b}A_{b,a}\\
 & =\exp\left(\Tr(v_{1}A)\right)A_{j,i}.
\end{align*}
 Then
\begin{align*}
\det(\pa_{v_{1}})\exp\left(\Tr(v_{1}A)\right) & =\sum_{\sm\in\Si_{r}}(\sgn\,\sm)\pa_{\sm(1),1}\pa_{\sm(2),2}\cdots\pa_{\sm(r),r}\exp\left(\Tr(v_{1}A)\right)\\
 & =\exp\left(\Tr(v_{1}A)\right)\sum_{\sm\in\Si_{r}}(\sgn\,\sm)A_{1,\sm(1)}A_{2,\sm(2)}\cdots A_{r,\sm(r)}\\
 & =\exp\left(\Tr(v_{1}A)\right)\det A.
\end{align*}
Thus the formula (\ref{eq:gamma-4}) is shown. Using this identity,
we have 
\begin{align*}
\cL^{(1,2)}U(z;\al) & =\det(\pa_{v_{1}})\det(z_{0}^{(1)},z_{1}^{(1)})^{-r}\exp\left(\Tr(v_{1}v_{2}^{-1})\right)(\det v_{2})^{\al_{0}^{(2)}}\\
 & =\det(z_{0}^{(1)},z_{1}^{(1)})^{-r}(\det v_{2})^{\al_{0}^{(2)}}\det(\pa_{v_{1}})\exp\left(\Tr(v_{1}v_{2}^{-1})\right)\\
 & =\det(z_{0}^{(1)},z_{1}^{(1)})^{-r}(\det v_{2})^{\al_{0}^{(2)}-1}\exp\left(\Tr(v_{1}v_{2}^{-1})\right)\\
 & =U(z;\al+\ep^{(1)}-\ep^{(2)}).
\end{align*}
Then the contiguity relation (\ref{eq:gamma-3-1}) can be written
as 
\[
F(\bx;\al)=\bb(\al_{0}^{(2)})F(\bx,\al+\ep^{(1)}-\ep^{(2)}),
\]
 and this in turn is written as 
\[
\G_{r}(a)=(a-1)(a-2)\cdots(a-r)\G_{r}(a-1),
\]
which is the formula to be shown.

\end{document}